\documentclass[a4paper,12pt]{article}
\usepackage{amsmath}
\usepackage{bm}
\usepackage{amssymb} 
\usepackage{graphicx}
\usepackage{setspace}
\usepackage{enumitem}                         
\usepackage{lscape}
 
\usepackage{amsthm}
\newtheorem{thm}{Theorem}[section]  
\newtheorem{lemma}[thm]{Lemma}
\newtheorem{cor}[thm]{Corollary}
\newtheorem{prp}[thm]{Proposition}
\theoremstyle{definition}
\newtheorem{dfn}[thm]{Definition}
\newtheorem{nttn}[thm]{Notation}

\theoremstyle{remark}
\newtheorem{remark}[thm]{Remark}

\numberwithin{equation}{section} 

\makeatletter  
  \@addtoreset{equation}{section}
\makeatother

\usepackage{minitoc}                 

\usepackage[usenames]{color}
  
  \newcommand{\red}[1]{\textcolor{red}{#1}}

\def\comment#1{ }

\author{
Yoshishige Haraoka \\
Hiroyuki Ochiai\\
Takeshi Sasaki\\
Masaaki Yoshida}

\title{\bf Shift operators of the Dotsenko-Fateev equation and its higher order versions}

\date{\today}

\begin{document} 

\maketitle


\begin{abstract}
  We find shift operators for the Dotsenko-Fateev equation, 
  which is a differential equation of order 3, and for the three Fuchsian differential equations of order 4, 5 and 6, respectively, which are connected with the Dotsenko-Fateev equation via addition and middle convolution. These shift operators are used to study reducible cases.
\end{abstract} 

\tableofcontents
\vfill

\noindent
{\bf Subjectclass}[2020]: Primary 34A30; Secondary 34M35, 33C05, 33C20, 34M03.

\noindent
{\bf Keywords}: Fuchsian differential equation, shift operator,
    reducibility,
  factorization, middle convolution, rigidity and accessory parameters,
  symmetry, hypergeometric differential equation, Dotsenko-Fateev equation.

\newpage
  
\section*{Introduction}
\addcontentsline{toc}{section}{\protect\numberline{}Introduction} 

In \cite{HOSY1}, we found a Fuchsian differential equation $H_6$ of order 6 with three singular points and with one accessory parameter, and its specialization $E_6$ with rich symmetry, that is defined by assigning the accessory parameter of $H_3$ a cubic polynomial of the local exponents. We found also the equations $E_5,E_4$ and $E_3$ of respective order 5, 4 and 3, that are connected with $E_6$ via addition and middle convolution (see \S \ref{GenAdd}). 

Here we briefly recall some terminology.
Given a differential equation with a local exponent $e$,
  which we denote by $E(e)$, a non-zero differential operator
  $P_{\pm}$ that sends solutions of $E(e)$ to those of $E(e\pm1)$
  is called the {\it shift operators} of $E(e)$
  for the shift of the parameter $e\to e\pm1$, and the operator is
  written as $P_{\pm}(e)$. 
These operators are important tools to see the structure of the space of solutions. If such operators $P_{\pm}$ exist, we define the operator
$Sv_{e}$ by $P_+(e-1)\circ P_-(e)$, which turns out to be a constant mod $E(e)$. 
We call such a constant the {\it S-value} for the shifts $e\to e\pm1$.
Note that when $Sv_{e}$ vanishes, $E(e)$ is {\it reducible}.

The Dotsenko-Fateev equation (\cite{DF}) is a codimension-2 specialization 
\footnote{For a Fuchsian equation $E$, a codimension-$k$ specialization of $E$ is an equation obtained by assuming $k$ linearly independent relations among the local exponents, apart from the Fuchsian relation.}
of $E_3$.  In this paper, this equation is denoted by $S\!E_3$. As $E_3$ is connected to $E_6, E_5$ and $E_4$, via addition and middle convolution,  
the equation $S\!E_3$ is lifted to codimension-2 specializations, say $S\!E_4, S\!E_5$ and $S\!E_6$,  of $E_4,E_5$ and $E_6$, respectively.

In this paper, we first find the shift operators of $S\!E_3$, and find those of  $S\!E_6,S\!E_5$ and $S\!E_4$.  These operators can be obtained theoretically from those of $S\!E_3$ by following the recipe in \cite{Osh} Chapter 11. However it is often difficult to carry out computation.
So we find the operators by other methods as we will explain in the following.

\par\medskip
This paper is organized as follows. In Section 1, the differential equations treated in this paper $E_3,\ S\!E_3,\ \dots,\ E_6,\ S\!E_6$ are presented. 

In order to define and study these equations, we need various tools of investigation, which we prepare in Section 2. 

In Section 3, the equation $E_3$ and the Dotsenko-Fateev equation $S\!E_3$ are introduced, and the shift operators of $S\!E_3$ are found. 

After reviewing the fundamental properties of $E_6$ in Section 4, a codimension-2 specialization $S\!E_6$ of $E_6$ is studied in Section 5. The study of the equations in this paper stemmed out from this equation. 
 The solutions of $S\!E_6$ admit an integral representation,
 which shows that the equation is obtained by a middle convolution
 of the Dotsenko-Fateev equation $S\!E_3$. 
 Apart from the shift operators coming from those of $E_6$, this equation has shift operators coming from those of  $S\!E_3$. Theoretically one can follow  \cite{Osh} Chapter 11, but 
 to get reasonably short expressions of them, we take another way
 (Proof of Theorem \ref{shiftopSE6}). Thus, we can compute S-values
 (Proposition \ref{S-valuesSE6}) and derive reducibility
 conditions (Theorem \ref{SE6redcond}).

 In Section 6, we study the factorizations of the equation $S\!E_6$ when it
 is reducible, which turns out to be more fruitful than those of $E_6$.
 It is related with the structure of  integral representation of solutions.

 After reviewing the fundamental properties of $E_5$ in Section 7, the codimension-2 specialization $S\!E_5$ of $E_5$ is studied in Section 8. The argument to get the shift operators for $S\!E_5$ is done
parallel to those of $S\!E_6$.
Refer to Theorem \ref{redcondSE5} for reducibility  conditions. 

After reviewing the fundamental properties of $E_4$ in Section 9, the codimension-2 specialization $S\!E_4$ of $E_4$ is studied in Section 10.

\par\smallskip
We acknowledge that we used the software Maple, especially
{\sl DEtools \!}-package for multiplication and division of
differential operators.
Interested readers may refer to our list of data written in text files of Maple format 
\footnote{ http://www.math.kobe-u.ac.jp/OpenXM/Math/HDFdata}
for the differential equations and the shift operators treated in this document.

\newpage

\section{The equations treated in this paper}\label{TheTen}

In this paper 
we treat Fuchsian ordinary differential equations
of order $2,\ \dots,\ 6$, with three singular points $\{0,1,\infty\}$
listed below.
For each equation, the coefficients are polynomials of the independent
variable $x$ and the local exponents at the three singular points.

\[ 
\begin{array}{lccccccccccc}
{\rm name}  &E_6&S\!E_6&E_5&S\!E_5&E_4&S\!E_4&{\rm ST}_4&E_3&S\!E_{3}&_{p}E_{p-1}&E_2 \\[1mm]
{\rm order} &6  &6     &5 &5     &4 &4     &4   &3  &3     &     p   &2\\
{\rm exponents}\quad &9&7&8&6&7&5&6&6&4&2p-1&3\\
{\rm acc.par.}&1&1&1&1&1&1&0&1&1&0&0
\end{array}
\]

\noindent
where 
name=name of the equation, order $=$ order of the equation,
exponents $=$ number of the free exponents,
acc.par. $=$ number of accessory parameters.
The notion of accessory parameter is given in \cite{HOSY1} \S 2.3.
${}_{p}E_{p-1}$ ($p=2, 3, \ldots$)
are the generalized hypergeometric equations,
and $E_2\ (={}_2E_1)$ is the Gauss hypergeometric equation.
$S\!E_n$ is a specialization of $E_n$.
${\rm ST}_4$ is studied in \cite{ST},
$S\!E_3$ is known as the Dotsenko-Fateev equation.

When we are studying a differential operator $E$, we often call $E$ a differential equation and speak about the solutions without assigning an unknown.

The {\it Riemann scheme} of an equation  is the table of local exponents
at the singular points. The {\it Fuchs relation} says that
the sum of all the exponents equals
\begin{equation}\label{Fuchsrelation}
  \frac12n(n-1)(m-2),\end{equation}
where $n$ is the order of the equation,
and $m$ is the number of singular points;
for our equations, $m=3$.
When an equation is determined uniquely by the local exponents,
it is said to be {\it free of accessory parameters} or {\it rigid}.
Though most of the equations we treat in this article have no such property,
for convenience sake, we often make use of the Riemann schemes;
one must be careful that the equation is not determined uniquely
  by the Riemann scheme.

\subsection{Table of the equations and their Riemann schemes}\label{TheTenTable}

For each equation $E$ of order $n$, we give the corresponding operator
with polynomial coefficients in $x$ of the form 
$$x^{n_0}(x-1)^{n_1}\partial^n+\text{\ lower\ order\ terms},\quad \partial=d/dx,$$
for some integers $n_0$ and $n_1$, such that the coefficients have no common factor, and call this operator also $E$. 

We {\it always assume} that local solutions corresponding
to local exponents with integral difference,
such as $0$, $1$, $2$; $s$, $s+1$, $s+2$,
has no logarithmic term, and the other exponents,
such as $e_1$, $e_2$, $\dots$,  are generic.
$R_n$ denotes the Riemann scheme of $E_n$.
\par\vskip 1pc

\noindent {\bf Table of the equations:}
\medskip

\begin{itemize}[leftmargin=*] \setlength{\itemsep}{5pt}
\item $E_6=E_6(e_1, \dots, e_9)=x^3(x-1)^3\partial^6+\cdots
  = T_0+T_1\partial+T_2\partial^2+T_3\partial^3,$
  \[R_6:\left(\begin{array}{lcccccc}
    x=0:&0&1&2&e_1&e_2&e_3\\
    x=1:&0&1&2&e_4&e_5&e_6\\
    x=\infty:&s&s+1&s+2&e_7&e_8&e_9\end{array}\right),
    \quad s=(6-e_1-\cdots-e_9)/3,\\
    \]
\[\def\arraystretch{1.1}
\begin{array}{lcl}
 T_0&=&(\theta +s+2)(\theta +s+1)(\theta +s)B_0,\quad B_0=(\theta +e_7)(\theta +e_8)(\theta +e_9),\quad \theta =x\partial,\\
 T_1&=&(\theta +s+2)(\theta +s+1)B_1,\quad B_1=T_{13}\theta ^3+T_{12}\theta ^2+T_{11}\theta +T_{10},\\
 T_2&=&(\theta +s+2)B_2,\quad B_2=T_{23}\theta ^3+T_{22}\theta ^2+T_{21}\theta +T_{20},\\
 T_3&=&-(\theta +3-e_1)(\theta +3-e_2)(\theta +3-e_3),
 \end{array}
 \] %
 where $T_{ij}$ are polynomials in $e_1$, $\dots$, $e_9$, see \S 3.
 The adjoint (see \S \ref{adjsec}) is given by
changing the parameters as
$$e_i\to 2-e_i\ (i=1,\dots,6),\quad e_j\to 1-e_j\ (j=7,8,9),\quad s\to -1-s.$$
Namely, the adjoint of $E_6(e_1, \dots, e_9)$
   is $E_6(2-e_1, 2-e_2, \dots, 2-e_6, 1-e_7, \dots, 1-e_9)$.
   
\item $S\!E_6$ is a specialization of $E_6$ characterized
  by the system of two equations:
 $$e_1-2e_2+e_3=e_4-2e_5+e_6=e_7-2e_8+e_9.$$
 It is often parameterized by $(a,b,c,g,p,q,r)$ as
 \addtolength{\arraycolsep}{-1pt}  
 \[\def\arraystretch{1.0}
 \begin{array}{lll}
e_1=p+r+1, &e_2=a+c+p+r+2, &e_3=2a+2c+g+ p+r+3,\\
e_4=q+r+1, &e_5=b+c+q+r+2, &e_6=2b+2c+g+ q+r+3,\\
e_7=-2c-pqr-1, &e_8=-\delta-c-pqrg-2, &e_9=-2\delta-pqrg-3,
 \end{array}
 \]
 \addtolength{\arraycolsep}{1pt}
and denoted as $S\!E_6(a,b,c,g,p,q,r)$,
 where $pqr=p+q+r$, $pqrg=pqr+g$, $\delta=a+b+c$. The
   parameter $s$ in $E_6$ turns out to be $-r$. The adjoint is given by
changing the parameters as
$$a\to -a-1,\ b\to -b-1,\ c\to -c-1,\ g\to-g,\ p\to -p-1,\ q\to -q-1,\ r\to -r+1.$$

\end{itemize}

\begin{itemize}[leftmargin=*]\setlength{\itemsep}{5pt}
\item $E_5=E_5(e_1,\dots,e_8)=x^3(x-1)^3\partial^5+\cdots=x\overline T_0+\overline T_1+\overline T_2\partial+\overline T_3\partial^2$,
\[R_5: \left(\begin{array}{lccccc}
  x=0&  0&1&e_1-1&e_2-1&e_3-1\\
  x=1&  0&1&e_4-1&e_5-1&e_6-1\\ 
  x=\infty&  1+s&2+s&3+s&e_7+1&e_8+1\end{array}\right),
    \qquad s=(6-e_1-\cdots-e_8)/3,
    \]
\[  \begin{array}{rcl}
   \overline T_0&=&(\theta +s+1)(\theta +s+2)(\theta +s+3)(\theta +e_7+1)(\theta +e_8+1), \\
   \overline T_1&=&(\theta +s+1)(\theta +s+2)B_{51},\quad B_{51}=B_1(e_9=0),\\
   \overline T_2&=&(\theta +s+2)B_{52},\quad B_{52}=B_2(e_9=0), \\
   \overline T_3&=&-(\theta +3-e_1)(\theta +3-e_2))(\theta +3-e_3).
\end{array} \]
This is obtained from $E_6$ by putting $e_9=0$,
and dividing from the right by $\partial$. The adjoint is given by
$$e_i-1\to 2-e_i \ (i=1,\dots,6),\ e_j+1\to 1-e_j\  (j=7,8), \ 1+s\to -s-1.$$

\item $S\!E_5$ is a specialization of $E_5$ characterized by the system
  of two equations:
$$e_1-2e_2+e_3=e_4-2e_5+e_6=e_7-2e_8.$$
It is often parameterized by $(a,b,c,g,p,q)$ as
  \[\def\arraystretch{1.0}
  \begin{array}{lll}
  e_1 = -2\delta - g - q - 2,
  &e_2 = -\delta-b - g - q - 1,
  &e_3 = -2b - q,\\
  e_4 = -2\delta - g - p - 2,
  &e_5 = -\delta-c - g - p - 1,
  &e_6 = -2a - p,\\
 &e_7 = 2a + 2b + g + 2,
  &e_8 = a + b + 1,\end{array}
  \]
  where $\delta=a+b+c$.  The adjoint is given as $S\!E_6$:
  $$a\to -a-1,\ b\to -b-1,\ c\to -c-1,\ g\to-g,\ p\to -p-1,\ q\to -q-1.$$
  
\item $E_4=E_4(e_1,\dots,e_7)=x^2(x-1)^2\partial^4+\cdots=\mathcal T_0+\mathcal T_1\partial+\mathcal T_2\partial^2$,

\[R_4:\left(\begin{array}{llcll}x=0:&0&1&e_1&e_2\\ x=1:&0&1&e_3&e_4\\ x=\infty:&e_8&e_5&e_6&e_7\end{array}\right),\quad e_1+\cdots+e_7+e_8+2=6,\]
 \[\def\arraystretch{1.1} \begin{array}{rcl}
    \mathcal T_0&=&(\theta +e_5)(\theta +e_6)(\theta +e_7)(\theta +e_8),\\
    \mathcal T_1&=&-2\theta ^3+\mathcal T_{12}\theta ^2+\mathcal T_{11}\theta +\mathcal T_{10},\\
    \mathcal T_{12}&=&e_1+e_2-e_5-e_6-e_7-e_8-5,\\
    \mathcal T_{11}&=&3(e_1+e_2)-e_1e_2+e_3e_4-e_5(e_6+e_7+e_8)-e_6e_7-e_6e_8-e_7e_8-8,\\
    \mathcal T_2&=&(\theta -e_1+2)(\theta -e_2+2),
 \end{array}
 \]
 where $\mathcal T_{10}$ is the accessory parameter,
 which is a polynomial in $e_1,\dots,e_7$ of degree 3, see \S \ref{E4ap1}.
The adjoint is given by
$$e_i\to1-e_i\ (i=1,\dots,8),$$
\item $S\!E_4$ is a specialization of $E_4$ characterized
  by the system of two equations:
  $$e_1-2e_2+1=e_3-2e_4+1=e_5-2e_6+e_7+e_8.$$
  We parameterize the 5 $(=7-2)$ free parameters by $a,b,c,g,u$ as:
$$\begin{array}{lll}e_1=2 a+2 c+g-u+2,&e_2=a+c-u+1,\\
e_3=2 b+2 c+g-u+2,\quad&e_4=b+c-u+1,\\
e_5=u+1,&e_6=u-a-b-2 c-g-1,\\e_7=u-2 c,&
e_8=u-2-g-2 c-2 b-2 a.\end{array}$$
The adjoint is given as $S\!E_6$:
$$a\to -a-1,\ b\to -b-1,\ c\to -c-1,\ g\to-g,\ u\to -u-1.$$

\item ${\rm ST}_4={\rm ST}_4(e_1,\dots,e_6)=x^2(x-1)^2\partial^4+\cdots=V_0+V_1\partial+V_2\partial^2,$ 
  \[\left(\begin{array}{llcll}x=0:&0&1&e_1&e_2\\ x=1:&0&1&e_3&e_4\\ x=\infty:&s&s+1&e_5&e_6\end{array}\right),\quad e_1+\cdots+e_6+2s+3=6,\]
\[\def\arraystretch{1.1}
\begin{array}{rcl}
 V_0 &=& (\theta +s+1)(\theta +s)(\theta +e_5)(\theta +e_6),\\
 V_1 &=& (\theta +s+1)\{-2\theta ^2+(e_1+e_2-e_5-e_6-4)\theta +\frac{1}{4}((e_6-e_5)^2 \\
  &&\qquad -(e_3-e_4)^2+(e_1-e_2)^2+2(e_1+e_2-3)(e_5+e_6+1)-1)\},\\
 V_2 &=& (\theta +2-e_1)(\theta +2-e_2).
\end{array}\]
This equation is rigid (free of accessory parameter),
and appears in Sasai-Tsuchiya \cite{ST}
studying an Okubo type equation of rank 4.       
The adjoint is given by
$$e_i\to -e_i+1\ (i=1.\dots,6),\quad s\to-s,$$
\item ${}_4E_3={}_4E_3(a_0,a_1,a_2,a_3;b_1,b_2,b_3)=x^3(x-1)\partial^4+\cdots$

  $= (\theta +a_0)\cdots(\theta +a_3)-(\theta +b_1)(\theta +b_2)(\theta +b_3)\partial$

  \qquad (the generalized hypergeometric equation),
\[ {}_4R_3:\left(\begin{array}{lcccc}
   x=0:&0&1-b_1&1-b_2&1-b_3\\
   x=1:&0&1   &2    &b_1+b_2+b_3-a_0-\cdots-a_3\\
   x=\infty:&a_0&a_1&a_2&a_3\end{array}\right).\]
This equation is rigid. The adjoint is given by 
$$a_i\to 1-a_i,\quad b_j\to 2-b_j.$$

\item $E_3=E_3(e_1,\dots,e_6)=x^2(x-1)^2\partial^3+\cdots=xS_n+S_0+S_1\partial,$

  \[R_3:\left(\begin{array}{llll}x=0:&0&e_1&e_2\\ x=1:&0&e_3&e_4\\ x=\infty:&e_7&e_5&e_6\end{array}\right),\quad e_1+\cdots+e_6+e_7=3,\]
\[\def\arraystretch{1.1}\setlength\arraycolsep{2pt} \begin{array}{rcl}
    S_n&=&(\theta +e_5)(\theta +e_6)(\theta +e_7),\\
    S_0&=& -2\theta ^3+(2e_1+2e_2+e_3+e_4-3)\theta ^2 \\
    &&\qquad +(-e_1e_2+(e_3-1)(e_4-1)-e_5e_6-(e_5+e_6)e_7)\theta +a_{00},\\
    S_1&=&(\theta -e_1+1)(\theta -e_2+1),
  \end{array}
  \]
  where $a_{00}$ is the accessory parameter,
  which is a polynomial in $e_1,\dots,e_7$ of degree 3, defined in Definition \ref{defA00}.
The adjoint is given by
$$e_i\to -e_i\ (i=1,\dots,4),\quad\ e_j\to2-e_j\ (j=5,6,7).$$
\item $S\!E_3$ (the Dotsenko-Fateev equation) is a specialization of $E_3$
  characterized by the system of two equations:
  $$2e_1-e_2=2e_3-e_4=-e_5+2e_6-e_7.$$
It is parameterized by $(a,b,c,g)$ as:
   $$\begin{array}{lll}   &e_1=a+c+1, &e_2=2e_1+g,\\
   &e_3=b+c+1, &e_4=2e_3+g,\\
          e_5=-2c,&     e_6=-(a+b+2c+g+1), &e_7=2e_6+g-e_5.\\
 \end{array}$$
The accessory parameter $a_{00}$ reduces to
  $$a_{00}=c(2a+2c+1+g)(2a+2b+2c+2+g).$$
\end{itemize}
The adjoint is given as: 
$$a \to -a - 1,\ b \to -b - 1,\ c \to -c - 1,\ g \to -g.$$

\subsection{Relation among the equations}\label{TheTenRel}

Under codimension-one conditions on the local exponents (a codimension-one condition means a
linear condition on the local exponents), the equations $E_6$, $\dots$, $S\!E_3$ factorize as follows:
\[\def\arraystretch{1.2}
\begin{array}{clll}
  {\rm equation}&&{\rm types\ of\ factorization}&{\rm reference}\\
  E_6 &\to& \{E_1,E_5\},\quad \{E_1,E_1,E_1,E_3\}&\S\ \ref{E6Red}\\
  S\!E_6 &\to& \{E_1, S\!E_5\},\quad  \{E_1,E_1,E_1,S\!E_3\},\quad  \{E_2,{\rm ST}_4\}&\S\ \ref{RedSE6fact}\\
  E_5 &\to& \{E_1,E_4\},\quad  \{E_1,E_1,E_3\}&\S\ \ref{E5Fact}\\
  S\!E_5 &\to& \{E_1,{\rm ST}_4\},\  \{E_1,S\!E_4\},\ \{E_1,\ {}_4E_3\},\\
  &&\{E_1,E_1,S\!E_3\}, \  \{E_2,\ {}_3E_2\}&\S\ \ref{SE5Red}\\
  E_4 &\to& \{E_1,E_3\}\quad &\S\ \ref{E4ap1Red}\\
  S\!E_4&\to& \{E_1,S\!E_3\}&\S\ \ref{newSE4}\\
  S\!E_3 &\to& \{E_1,E_2\}&\S\ \ref{red_cases}\\
\end{array}\]
where $E_1$ stands symbolically for an operator of order 1.
In the above table, $E_6 \to \{E_1,E_5\}$, for example, reads:
$$\text{under codimension-one conditions, $E_6$ factors as}
\ F_1\circ F_5 {\rm \ and \ }F'_5\circ F'_1,
\footnote{Composition of two differential operators $P$ and $Q$ is
  denoted by $P\circ Q$; we often  write it as $PQ$.}$$
where $F_1$ (and $F'_1)$ are operators of order 1, and $F_5$ (and $F'_5$) are essentially  equal (see Definition \ref{essentially}) to $E_5$.
For example, we have
\[ \def\arraystretch{1.1} \begin{array}{lcl}
  E_6(e_9=1)&=&\partial\circ E_5(e_1+1,\dots,e_6+1,e_7-1,e_8-1),\\
  E_6(e_9=0)&=&E_5(e_1,\dots,e_6,e_7,e_8)\circ \partial.
\end{array}\]
For $E_3$, no codimension-one condition to be reducible is found
  to the authors.


\section{Generalities}\label{Gen}

We prepare some generalities that we need in the following sections.
\subsection{$(x,\partial)$-form and $(x,\theta,\partial)$-form}\label{Genxz}
Given a differential operator $P=a_n(x)\partial^n+\cdots\in\mathbb{C}[x][\partial]$ in $(x,\partial)$-form. Rewrite each term as
  $$x^i\partial^j=x^{i-j}(x^j\partial^j),\quad i\ge j,\qquad
x^i\partial^j=(x^i\partial^i)\partial^{j-i},\quad i\le j,
    $$
and substitute 
  $$x^i\partial^i=\theta(\theta-1)\cdots(\theta-i+1),\quad i\ge1,\quad \theta=x\partial.$$
Then $P$  can be written uniquely in $(x,\theta,\partial)$-form:
  $$P=x^qP_{-q}+\cdots+xP_{-1}+P_0+P_1\partial+\cdots+P_p\partial^p,\ \ p\le n,\ q\ge0,$$
  where $P_i$ is a polynomial in $\theta$.
  \subsection{Adjoint equation} \label{adjsec}

  The adjoint $P^*$ of $P=\sum a_j(x)\partial^i$ is defined as $P^*=\sum(-)^j\partial^j\circ a_j(x).$
An equation $E=E(e)$ with a set of parameters $e$ is said to
  have adjoint symmetry if $E(e)^*=E({\rm adj}(e))$
  for a linear transformation ${\rm adj}$ on the set $e$.
  Every equation $E$ appeared in \S1 has adjoint symmetry;
  for example in the case of $S\!E_3$, ${\rm adj}(a,b,c,g)
  =(-a-1, -b-1, -c-1, -g)$. We refer to \cite[\S 2]{HOSY1}.
  
\subsection{Addition and middle convolution }\label{GenAdd}

We recall the operations called {\it addition} and
  {\it middle convolution} referring to \cite{Hara, Osh}.

\begin{dfn}\label{GenAddRL}
  For a function $u(x)$, {\it Riemann-Liouville transformation} of $u$
  with parameter $\mu$ is defined as the function in $x$:
 $$I^{\mu}_{\gamma}(u)(x)=\frac1{\Gamma(\mu)}\int_\gamma u(t)(x-t)^{\mu-1}\, dt,$$
 where $\gamma$ is a cycle.
\end{dfn}

\begin{dfn}\label{GenAddAdd}
 For a linear differential operator $P$ in $x$ and a function $f$ in $x$,
the operator called  the {\it addition} by $f$ is defined as
 $${\rm Ad}(f)P:=f\circ P\circ f^{-1}.$$
\end{dfn}

\begin{dfn}\label{defmc}\label{GenAddMid}
 If $u$ is a solution of a linear differential equation $P$, then
the Riemann-Liouville transformation $I^{\mu}_{\gamma}(u)$ turns out to be a solution of
the differential equation $mc_{\mu}(P)$, called the {\it middle convolution of $P$ with parameter $\mu$}.
\end{dfn}

 The equation $mc_{\mu}(P)$ is obtained as follows:
 Multiply $P$ by $\partial^k$ with sufficiently large positive integer $k$
 from the left so that $\partial^kP$ can be written
 as a linear combination of $z^i\circ\partial^j$, where $z=x\partial$.
Then operate ${\rm Ad}(\partial^{-\mu})$, and divide the result by $\partial$
from the left as many times as possible, say $\ell$.
Namely we have
$$
 {\rm Ad}(\partial^{-\mu})(\partial^kP)=\partial^\ell\circ mc_{\mu}(P).
$$
 (The result is independent of $k$.) In computing the left hand side,
the formulas
$$
 {\rm Ad}(\partial^{-\mu})\partial=\partial,\quad
 {\rm Ad}(\partial^{-\mu})z=z-\mu
$$
are useful. The middle convolution has the additive property
$$
 mc_0={\rm id.},\quad
 mc_{\mu}\circ mc_{\mu'}=mc_{\mu+\mu'},
$$
and so $mc_{\mu}$ is invertible:
$$
(mc_{\mu})^{-1}=mc_{-\mu}.
$$

\noindent
For an operator $P$ with singular points $0$, $1$, $\infty$, set
$$\begin{array}{lcl}
  d&=& ({\rm mult\ of\ }0{\ \rm in\ the\ exponents\ at\ }x=0)\\
  && +({\rm mult\ of\ }0{\ \rm in\ the\ exponents\ at\ }x=1)\\
  && +({\rm mult\ of\ }\mu{\ \rm in\ the\ exponents\ at\ }x=\infty)-{\rm order}(P),\end{array}
$$     
where multiplicity (abbreviated as mult) is counted mod 1.
Here and in the following,
${\rm order}(P)$ denotes the order of the operator $P$. 
Then we have
  $${\rm order}(mc_{\mu}(P))={\rm order}(P)-d.$$
For example, the exponents at infinity of $E_6$ are
$\{s,\, s+1,\,  s+2,\, e_7,\, e_8,\, e_9\}$,
so the parameter $\mu$ can be equal ( mod 1) to none of them,
or to $e_j\,(j=7,\, 8,\, 9)$, or equal to $s$. Thus
  \[{\rm (mult\ of\ } \mu\ {\rm in\ the\ exponents\ at\ } x=\infty)=
  0\ {\rm or}\ 1\ {\rm or}\ 3
  \]
and
  $$d=3+3+(0{\ \rm or\ }1{\ \rm or\ }3)-6=0{\ \rm or\ }1{\ \rm or\ }3.$$
Similarly, we have

\begin{itemize}[leftmargin=36pt]
\item if $P=E_3$, then $d=1+1+(0{\ \rm or\ 1})-3=-1{\ \rm or\ }0$.
  If we perform an addition first to change the local exponent 0 of $x=0$
  or/and $x=1$ non-zero, then $d=-3,-2{\ \rm or\ }-1$.
  See \cite{HOSY1} \S 3.0.1.

\item if $P=E_2$, then $d=1+1+(0{\ \rm or\ }1)-2=0{\ \rm or\ }1$.
  Thus any middle convolution of $E_2$ is again a Gauss operator
  or a 1st order operator. But if we perform an addition first
  to change the local exponent 0 of $x=0$ or/and $x=1$ non-zero,
  then $d=-2,-1{\ \rm or\ }0$. See\cite{HOSY1} \S 3.1.
\end{itemize}

\noindent
It is known  that middle convolutions do not change the number of accessory parameters. 

\subsection{Shift operators, shift relations and S-values}\label{GenShift}

We recall the definition of shift operators, shift relations,
  and $S$-values from \cite[\S 4]{HOSY1}.

Let $E(a)\in D=\mathbb{C}(x)[\partial]$ be an operator of order $n$
with a parameter $a$. Assume that it is irreducible for generic $a$.
Let $S$ denote the set of singularities of $E(a)$ in this subsection.

\subsubsection{Shift operators and shift relations}\label{GenShiftShift}
\begin{dfn}\label{DefShift}In general,
  let  ${\rm Sol}(E(a))$ be the solution space of $E(a)$.
  A non-zero operator $P\in D$ of order lower than $n$
  sending ${\rm Sol}(E(a))$ to ${\rm Sol}(E(a+1))$, is called
  a {\it shift operator} for the parameter change $a\to a+1$,
  and denoted by $P_{a+}$. A shift operator for $a\to a-1$,
  is denoted by $P=P_{a-}$.
\end{dfn}
Since we assume $P\in D$, we have

\begin{lemma}\label{pi1morphsim}The shift operators
  are $\pi_1(\mathbb{C}-S)$-morphism, {\it i.e.},
  they commute with the $\pi_1(\mathbb{C}-S)$-action.
\end{lemma}
Suppose a shift operator $P_{a+}$ exists.
Since $E(a+1)\circ P_{a+}(a)$ is divisible from right by $E(a)$,
there is an operator $Q_{a+}(a)$ satisfying the {\it shift relation}:
  $$(EPQE):\quad E(a+1)\circ P_{a+}(a)=Q_{a+}(a)\circ E(a).$$
Conversely, if there is a pair of non-zero operators $(P_{a+}(a), Q_{a+}(a))$ of order smaller than $n$ satisfying this relation, then $P_{a+}$ is a shift operator for $a \to a + 1.$ We often call also
the pair $(P_{a+}(a), Q_{a+}(a))$ the shift operator for $a \to a + 1.$
\begin{prp}\label{expofQ} If an operator $E(e)$ with the adjoint symmetry
  $E(e)^*=\\E (adj(e))$, where $e$ is the set of parameters,
  admits a shift relation $E(\sigma(e))\circ P=Q\circ E(e)$ for a shift $\sigma$, then 
  $$Q=(-)^\nu P(adj\circ \sigma(e))^*,\quad \nu={\rm order}(P).$$
\end{prp}

By Schur's lemma, we have
\begin{prp}\label{inverseofshiftop}If $E(a)$ is irreducible and $P_{a+}$
  $($resp. $P_{a-}\,)$ exists then the inverse operator $P_{a-}$
  $($resp. $P_{a+}\,)$ exists. It is unique up to multiplicative constant (independent of $(x,\partial)$). 
\end{prp}

\begin{remark}
  When a differential {\it equation} in question is $Eu=0$
($u$ is indeterminate), by multiplying a non-zero polynomial to the {\it operator} $E$, we can assume that $E$ has no poles. However, the coefficients of shift operators may have poles as functions of $x$.
  \end{remark}

\subsubsection{S-values}\label{GenShiftS}

When the shift operators exist for the shifts $a\to a\pm1$,
consider a composition of shift operators:
  $$P_{a+}(a-1)\circ P_{a-}(a):{\rm Sol}(E(a))\to{\rm Sol}(E(a-1))\to{\rm Sol}(E(a)).$$
For a generic $a$, again by Schur's lemma,
there is an operator $R$ such that
  $$P_{a+}(a-1)\circ P_{a-}(a)-R\circ E(a)$$
is a constant (times the identity). 

\begin{dfn}This constant will be called the {\it S-value} for $a$,
  and will be denoted by $Sv_{a-}(a)$.   
\end{dfn}
We often write $Sv_{a-}(a)\equiv P_{a+}(a-1)\circ P_{a-}(a)$;  $\equiv$
means {\it modulo $E(a)$}, and when
we are discussing S-values, we forget the identity map.

\begin{prp}\label{S-valuesprop210} The composition
  $$P_{a-}(a+1)\circ P_{a+}(a):{\rm Sol}(E(a))\to{\rm Sol}(E(a+1))\to{\rm Sol}(E(a))$$
  is a constant modulo $E(a)$. Denote the constant as $Sv_{a+}(a)$. Then
  $$Sv_{a-}(a)=Sv_{a+}(a-1).$$
\end{prp}

\begin{prp}\label{Sred}   
If $Sv_{a+}(a)=0$ $($resp. $Sv_{a-}(a)=0)$, then $E(a)$ and $E(a+1)$ $($resp. $E(a-1))$ are reducible. If  $Sv_{a+}(a)\not=0$ $($resp. $Sv_{a-}(a)\not=0)$, then $P_{a+}$ $($resp. $P_{a-})$ gives an isomorphism: ${\rm Sol}(E(a))\to {\rm Sol}(E(a+1))$ $($resp. ${\rm Sol}(E(a))\to {\rm Sol}(E(a-1)))$ as $\pi_1(\mathbb{C}-\{0,1\})$-modules.
\end{prp}

\begin{thm}\label{red_atoap}
Assume $E$ and $E'$ are connected by the shift relation $E'P=QE$.
  If $E$ is reducible, so is $E'$. If $E'$ is reducible, so is $E$.
\end{thm}

\subsection{Reducibility type and shift operators }\label{GenRed}

We discuss factorization of Fuchsian operators in $D=\mathbb{C}(x)[\partial]$.
We refer to \cite[\S 6]{HOSY1}.

\begin{dfn}When $E\in D$ is reducible and factorizes as
  $$E=F_1\circ \cdots\circ F_r,\quad F_j\in D,\quad 0<{\rm order}(F_j)=n_j,\ (j=1, \dots, r),$$
  we say $E$ is {\it reducible of type} $[n_1, \dots, n_r]$;
we sometimes call $[n_1, \dots, n_r]$ the {\it type of factors}.
We often forget commas, for example, we write [23] in place of [2, 3].
When only a set of factors matters,
we say $E$ is {\it reducible of type} $\{n_1, \dots, n_r\}$. 
\end{dfn}

\begin{lemma}\label{unique}The way of decomposition is not unique. But if the factors are irreducible, the set of orders of the factors, including multiplicity, is unique: 
if $P_1...P_n=Q_1...Q_m$ then $\{{\rm ord}(P_1),....{\rm ord}(P_n)\}=\{{\rm ord}(Q_1),....{\rm ord}(Q_m)\}.$

\begin{proof} Let $S$ be the solution space of $P_i...P_n$. The space $S_i:=P_{i+1}...P_n(S)$ is ${\rm ord}(P_i)$-dimensional.
  $Q_m(S_i)$ is $\{0\}$ or ${\rm ord}(P_i)$-dimensional. If latter is the case, then $Q_{m-1}Q_m(S_i)$  is $\{0\}$ or ${\rm ord}(P_i)$-dimensional $...$ Thus $\{{\rm ord}(P_1),....{\rm ord}(P_n)\}\subset\{{\rm ord}(Q_1),....{\rm ord}(Q_m)\}$.
\end{proof}
\end{lemma}

\begin{prp}\label{invsubsp}
  $E$ admits a factorization  $F_1\circ \cdots\circ F_r$
  of type $[n_1, \dots, n_r]$ if and only if ${\rm Sol}(E)$
  has monodromy invariant subspaces
  $${\rm Sol}(E)=S_1\supset S_2\supset\cdots\supset S_r,$$
  with
  $$\dim S_1/S_2=n_1,\ \dim S_2/S_3=n_2,\dots,\  \dim S_r=n_r.$$
\end{prp}

\begin{nttn}\label{nttn0} If $F_j\ (j=1,\dots,r)$ have no singular points other
than the singular points of $E$, we write $$[n_1, \dots, n_r]A0.$$
\end{nttn}

Even if the equation $E$ has singularity only at $S=\{0, 1, \infty\}$,
the factors may have singularities out of $S$.

\begin{prp}\label{apparentsing} If $E$ has singularity only at $S$, then the singular points of $F_1$ and $F_r$ out of $S$ are apparent.
\end{prp}
\begin{remark} The way of factorization is far from unique: in fact,
an operator can have different types of factorization such as 
the shift relation $E'\circ P=Q\circ E$ and the factorizations
\begin{gather*}
  A\circ B = (A\circ f)\circ (f^{-1}\circ B),\ f\in\mathbb{C}(x),\ f\not=0,\\
  \partial^2=\left(\partial+\frac1{x-c}\right)\circ
  \left(\partial-\frac1{x-c}\right),\ c\in \mathbb{C}.
\end{gather*}
Therefore, when we discuss the singularity of the factors of a decomposition,
we usually choose the factors so that they have least number of singular points.
\end{remark}

\noindent
Proposition \ref{Sred} and Proposition \ref{invsubsp} lead to

\begin{prp}\label{FactorType}
  Assume $E(a)$ and $E(a+1)$ $(resp.\ E(a-1))$ are connected
  by a shift relation. If $Sv_{+}\not=0$ $(resp.\ Sv_{-}\not=0)$,
  then $E(a)$ and $E(a+1)$ $(resp.\ E(a-1))$ admit the factorization of
  the same type.
\end{prp}


\section{$S\!E_3$:  a specialization of $E_3$}

\subsection{The Dotsenko-Fateev equation} 
The Dotsenko-Fateev equation is originally found as a differential equation satisfied by functions in $x$ defined by the integral
\begin{equation}\label{intrep}\int\omega(x),\quad \omega(x):=\prod_{i=1,2} t_i^a(t_i-1)^b(t_i-x)^c\cdot (t_1-t_2)^g\ dt_1\wedge dt_2
\end{equation}
Consider in the real $(t_1,t_2)$-plane the arrangement of seven lines:
  $$ \prod_{i=1,2}t_i(t_i-1)(t_i-x)\cdot(t_1-t_2)=0.$$
Since the number of bounded chambers cut out by this arrangement is 6,
if the exponents of the integrand is generic,
functions defined by the above integral would satisfy
a differential equation of order 6 (cf. \cite{AK}).
But since the integrand is invariant
under the change $t_1\leftrightarrow t_2$,
and the number of bounded chambers modulo this change
is 3, the functions defined by the above integral satisfy
an equation of order $3$, which is the  Dotsenko-Fateev equation.

The Dotsenko-Fateev operator (\cite{DF}) is a Fuchsian operator of order 3
and is defined as
\begin{equation}\label{dfequaion}
  D\!F= D\!F(a,b,c,g;x)=x^2(x-1)^2\partial^3+D\!F_2\partial^2+D\!F_1\partial+D\!F_0,
\end{equation}
where 
\begin{eqnarray*} 
 D\!F_2 &=& -(x-1)x(3ax+3bx+6cx+2gx-3a-3c-g), \\
 D\!F_1 &=& 2a^2x^2+4abx^2+12acx^2+3agx^2+2b^2x^2+12bcx^2+3bgx^2\\
 && +12c^2x^2+8cgx^2+g^2x^2-4a^2x-4abx-16acx-4agx+ax^2-8bcx\\
 && -2bgx+bx^2-12c^2x-8cgx +6cx^2 -g^2x+gx^2+2a^2+4ac+ag\\
 && -2ax+2c^2+cg-6cx-gx+a+c, \\
 D\!F_0 &=& c(2a+2b+2c+g+2)(-(2a+2b+4c+2g+2)x+2a+2c+g+1).
\end{eqnarray*} 
The Riemann scheme is 
\begin{equation} \label{RDF}
R_{D\!F}=
\left( \begin{array}{cccc}
  x=0 & 0&e_1=a+c+1&e_2=2a+2c+g+2\\
  x=1 & 0&e_3=b+c+1&e_4=2b+2c+g+2\\
  x=\infty&\ e_5=-2c\ &\ e_6=-\delta-c-g-1\ &\ e_7=-2\delta-g-2
\end{array} \right),
\end{equation}
where $\delta=a+b+c$.
The accessory parameter is the constant term of $D\!F_0$:
$$c(2\delta+g+2)(2a+2c+g+1).$$
Note that
\begin{equation}-2e_1+e_2=-2e_3+e_4=e_5-2e_6+e_7 \quad (=g).
\end{equation}
The adjoint is given by
$$a \to -a - 1,\ b \to -b - 1,\ c \to -c - 1,\ g \to -g.$$
Easy to check the following symmetries:
\begin{equation} \label{symm}D\!F(a,b,c,g;x)=-D\!F(b,a,c,g;1-x)
=-x^{e_7}D\!F(a,c,b,g;x^{-1})\circ x^{-e_7}.
\end{equation}

\subsection{Equations $H_3$,  $E_3$ and $S\!E_3$} \label{$G3$}

We review the Fuchsian differential equation $E_3$ of order 3 with three singular points, and with the unique accessory parameter specified as a cubic polynomial of the local exponents.
Any Fuchsian differential equation of order 3 with the Riemann scheme
$$R_3=\left(\begin{array}{clll}x=0&0&e_1&e_2\\x=1&0&e_3&e_4\\ x=\infty&e_5&e_6&e_7\end{array}\right)\quad e_1+\cdots+e_7=3
$$
can be written in $(x,\partial)$-form as
$$H_3: a_3(x)\partial^3+a_2(x)\partial^2+a_1(x)\partial+a_0(x),\quad \partial=\frac{d}{dx},$$
where
\[\def\arraystretch{1.2}\setlength\arraycolsep{2pt} \begin{array}{rcl}
  a_3(x)&=&x^2(x-1)^2,\ \ a_2(x)=x(x-1)(a_{21}x+a_{20}),\\
  a_1(x)&=&a_{12}x^2+a_{11}x+a_{10},\ \ a_0(x)=a_{01}x+a_{00}, \\
  a_{21}&=&6-e_1-e_2-e_3-e_4,\ \ a_{20}=-3+e_1+e_2,\\
  a_{12}&=&(e_5+e_6+1)e_7+(e_6+1)(e_5+1), \\
  a_{11}&=&-(e_5+e_6)e_7+(-e_1e_2+e_3e_4-e_5e_6+2e_1+2e_2-4),\\
  a_{10}&=&(e_1-1)(e_2-1),\ a_{01}=\ e_5e_6e_7,\  
\end{array}
\]
and in $(x,\theta ,\partial)$-form as 
$$xS_n(\theta)+S_0(\theta)+S_1(\theta)\circ\partial,\quad \theta =x\partial,$$
where
\begin{equation} \label{SE3-theta-form}
\begin{array}{rcl}
  S_n(\theta)&=&(\theta +e_5)(\theta +e_6)(\theta +e_7), \\
  S_0(\theta)&=&-2\theta^3+(2e_1+2e_2+e_3+e_4-3)\theta^2 \\
  &&\quad +(-e_1e_2+(e_3-1)(e_4-1)-e_5e_6-e_6e_7-e_7e_5)\theta +a_{00},\ \\
  S_1(\theta)&=&(\theta -e_1+1)(\theta -e_2+1).
\end{array}  
\end{equation}
The coefficient $a_{00}$ does not affect the local exponents. In this sense one can call this coefficient the accessory parameter.
\begin{dfn}\label{defA00}$E_3$ is the equation $H_3$ with the Riemann scheme $R_3$
  and with the accessory parameter $a_{00}=A_{00},$
where
  \begin{equation}\label{accpara00}
  \def\arraystretch{1.1}\setlength\arraycolsep{2pt} \begin{array}{rcl}
  54A_{00}&=& -4(e_1+e_2-e_3-e_4)^3-27e_5e_6e_7\\
  &&+9(e_1+e_2-e_3-e_4)(e_5e_6+e_5e_7+e_6e_7-2)\\
  &&  +9e_1e_2(e_1+e_2-1)+18(e_1+e_2-1)(e_3^2+e_3e_4+e_4^2)\\
  &&-9e_3e_4(e_3+e_4-1)-18(e_3+e_4-1)(e_1^2+e_1e_2+e_2^2).
  \end{array}
\end{equation}
\end{dfn}
\begin{remark}The equation $E_3$ is obtained from  the equation $E_6$, studied in \cite{HOSY1}, via addition and middle convolution. $E_6$ has nice symmetries. Though no shift operator is found for $E_3$, it inherits some of them:
\end{remark}
\begin{prp}\label{symmE3}

\begin{itemize}
\item Adjoint symmetry:
  $$E_3(e)^*=E_3(-e_1,-e_2,-e_3,-e_4,2-e_5,2-e_6).$$
\item $(x\to1-x)$-symmetry: $$E_3(e_1,\dots,e_6)|_{x\to1-x}=E_3(e_3,e_4,e_1,e_2,e_5,e_6),$$
\item $(x\to1/x)$-symmetry: $$x^{-e_7}E_3(e_1,\dots,e_6))|_{x\to1/x}\circ x^{e_7}=E_3(e_5-e_7,e_6-e_7,e_3,e_4,e_1+e_7,e_2+e_7),$$
  where $E_3|_{x\to1-x}$ and $E_3|_{x\to1/x}$ are $E_3$ after the coordinate changes $x\to1-x$ and $x\to1/x$, respectively.
\end{itemize}
\end{prp}

\begin{dfn}The equation $E_3$ with a codimension-2 condition (system of two equations) $$2e_1-e_2=2e_3-e_4=-e_5+2e_6-e_7.$$
is denoted by $S\!E_3$. 
\end{dfn}

\begin{thm}\label{DFequalsSE3}The equation $S\!E_3$ is transformed to the Dotsenko-Fateev operator $D\!F$ by the change (cf. Riemann scheme $R_{D\!F}$)
  $$\begin{array}{lllll}&&e_1=a+c+1,\ \ &e_2=2a+2c+g+2,\\
    &&e_3=b+c+1,\ \ &e_4=2b+2c+g+2,\\
    e_5&=-2c,\ \ &e_6=-\delta-c-g-1\ \ &e_7=-2\delta-g-2,\qquad \delta=a+b+c.\end{array}$$
\end{thm}

\begin{proof} Substitute the above in the expression of $A_{00}$ given in \eqref{accpara00}
  to find it turns out to be $c(2a+2c+g+1)(2a+2b+2c+g+2)$. \end{proof} 
In the following we parameterize the local exponents of $S\!E_3$ by $\{a,b,c,g\}$ as above.

\subsection{Shift operators of $S\!E_3$} \label{se3op}

\begin{thm}\label{S13shiftopSE3}The equation $S\!E_3$ admits shift operators for the respective shifts
  $$a\pm:a\to a\pm1,\ b\pm: b\to b\pm1,\ c\pm:c\to c\pm1 {\rm \quad and\quad} g\pm:g\to g\pm2.$$
\end{thm}
When the parameter $e$ increases (resp. decreases) by one
(by two for the parameter $e=g$),
we denote the operators $P$ and $Q$ by $P_{e+}$ and $Q_{e+}$ (resp.
$P_{e-}$ and $Q_{e-}$) in the following. Thus, we have
$P_{a\pm},P_{b\pm},P_{c\pm},P_{g\pm}$ as for operators $P$.

\begin{remark} Shift operators do not exist for the shifts $g\to g\pm1$ for the following reason. In the integral representation \eqref{intrep}
  of the solutions of $S\!E_3$, the exponent of the factor $(t_1-t_2)$
  of the integrand is $g$.
  Since we divide the expression by the symmetry $t_1\leftrightarrow t_2$,
  the exponent is considered to be $g/2$.
  So the monodromy group of $S\!E_3(g)$
  and that of $S\!E_3(g+1)$ are different and, in general, there is no
  non-zero differential operator with coefficients in $\mathbb{C}(x)$
  maps the solution space of  $S\!E_3(g)$ to that of $S\!E_3(g+1)$. 
\end{remark} 

\subsubsection{Shift operators $P_{a+},P_{c+},P_{g+}$}

The shift operator, for example $P_{a+}$, can be characterized as a non-zero operator satisfying the shift relation (see \S \ref{GenShiftShift}):
$$(EPQE):\quad S\!E_3(a\to a+1)\circ P_{a+}=Q_{a+}\circ S\!E_3$$
for a certain operator $Q_{a+}$. Assume that $P_{a+}$ and $Q_{a+}$ have the $(x,\theta)$-form
$$
P=x^2P_{nn}(\theta )+xP_n(\theta)+P_0(\theta ),\quad
Q=x^2Q_{nn}(\theta )+xQ_n(\theta)+Q_0(\theta ),$$
where  $P_{nn},\dots,Q_0$ are quadratic polynomials in $\theta:=x\partial$. 
Write $S\!E_3$ in $(x,\theta,\partial)$-form 
$$xS_n+S_0+S_1\partial,\quad S_*{\rm\ are\ in\ }\eqref{SE3-theta-form}$$ 
and expand both sides of the equation $(EPQE)$ and, by using the formulae
$$\theta x=x(\theta +1),\quad \partial\theta =(\theta +1)\partial,$$
rewrite the equation in $(x,\theta,\partial)$-form:
$$x^3S_n(a+1,\theta +2)P_{nn}+\cdots+S_1(a+1)P_0(\theta+1)\partial
=x^3Q_{nn}(\theta+1)S_n+\cdots+Q_0S_1\partial,$$
where $S_n(a+1,\theta +2):=S_n(a\to a+1,\theta\to\theta +2)$, etc. 
From the equation $S_n(a+1,\theta +2)P_{nn}=Q_{nn}(\theta+1)S_n$, 
immediate to see 
$$P_{nn}  = (\theta +e_5)(\theta +e_6),\quad Q_{nn} = (\theta +e_5+1)(\theta +e_6),$$and from $S_1(a+1)P_0(\theta+1)=Q_0S_1$, to see 
$$P_{0}  = C(\theta -e_1)(\theta -e_2),\quad Q_{0} = C(\theta -e_1)(\theta -e_2-1)\quad C: {\rm constant}.$$
Now $P_n$ and $Q_n$ can be fixed by setting $C=1$.
Summing up we get

\begin{prp} $P_{a+}=x^2P_{nn}+xP_{n}+P_0,\ Q_{a+}=x^2Q_{nn}+xQ_{n}+Q_0$, where
  $$
  \begin{array}{lllllll} P_{nn}  &=& (\theta +e_5)(\theta +e_6), \quad Q_{nn}  = (\theta +e_5+1)(\theta +e_6), \\
   P_{n}  &=& -2\theta^2+(4+4a+7c+2g+b)\theta -c(6a + 2b + 6c + 3g + 6),\\
   Q_{n}  &=& P_n+2+2a+g, \\
   P_{0}  &=& (\theta -e_1)(\theta -e_2), \quad  
   Q_{0}  = (\theta -e_1)(\theta -e_2-1).
  \end{array}
  $$
To make the expression shorter, we also used notation for the exponents
$$e_1=a+c+1,\ \dots,\  e_7=-2a-2b-2c-g-2.$$
\end{prp}

Same for $(P_{b+},Q_{b+}),(P_{c+},Q_{c+})$ and $(P_{g+},Q_{g+})$. Since the coordinate change $x\to1-x$ just exchanges $a$ and $b$, in the following, we omit the description for $P_{b\pm}$. Since $Q$'s are expressible by $P$'s (Proposition \ref{expofQ}), we also omit $Q_{c+}$ and $Q_{g+}$.

\begin{prp}$P_{c+}=x^2P_{nn}+xP_{n}+P_0$, where
$$\begin{array}{llll}
 P_{nn}  &=& \theta^2+(-5-3a-6c-2g-3b)\theta  +10ca+10cb+3ag+7cg+8 \\
  & &\quad +10c^2+18c+2a^2+8a+2b^2+8b+g^2+6g+4ab+3bg, \\
 P_{n}  &=&-2\theta^2+(8+6a+3b+9c+3g)\theta  -6cb-14ca-8-7cg  \\
  & & \quad -2bg-4ab-4ag-18c-10c^2-12a-6g-4b-4a^2-g^2, \\
 P_{0}  &=& (\theta -e_1)(\theta -e_2), \\
\end{array}$$
and $P_{g+}=x^2P_{nn}+xP_{n}+P_0$, where    
$$  \begin{array}{lllll}
P_{nn} &=& (k_1\theta +k_0)(\theta +e_5), \\
P_{n} &=& (6g+12+4b+4a+4c)\theta^2 +(-6g^2-12ab-4b^2-16cb-24-8a^2\\
    && \quad -24cg-8bg-20ca-24g-46c-12c^2-16ag-32a-18b)\theta  \\
    &&\quad  +8cb^2+28cb+36ca+12cbg+20acg+16abc+36c^2 \\
    &&\quad  +8c^3+16bc^2+16c^2a+20c^2g+30cg+8ca^2+8cg^2+28c, \\
P_{0} &=& (m_1(\theta -1)+m_0)(\theta -e_2), \\
\end{array}$$
and
$$\begin{array}{lll}k_1=-2\delta-3g-6,\quad    k_0 = 4(\delta+g+2)^2+2c(g+1)-g-2,\\  
\delta=a+b+c,\quad    m_1=k_1,\quad    m_0 = -2\delta-2b-8-g^2-2bg-6g.
\end{array} $$
\end{prp}

\subsubsection{Shift operators $P_{a-},P_{c-},P_{g-}$}

For the shifts $a-, c-$ and $g-$, instead of solving the equation $(EPQE)$,  we solve an easier equation, in case $P_{a-}$ for example,
$$P_{a-}(a\to a+1)\circ P_{a+}\equiv {\rm constant\quad modulo\ } S\!E_3.$$ 
For the shifts $a-,c-,g-$, the operators are of the following form:
\[\def\arraystretch{1.1} \begin{array}{ll}
    P_{a-}:&\displaystyle   (x-1)^2\ \partial^2+  (x-1)\frac{{\rm Poly}_1}{x}\partial  + \frac{{\rm Poly}_1}x,\\[3mm]
    P_{c-}:&\displaystyle      \partial^2    +\frac{{\rm Poly}_1}{x(x-1)}\partial  + \frac{{\rm Poly}_0}{x(x-1)},\\[3mm]
    P_{g-}:&\displaystyle   {\rm Poly}_2\ \partial^2+\frac{{\rm Poly}_3}{x(x-1)}\partial  + \frac{{\rm Poly}_2}{x(x-1)},
\end{array}\]
where $ {\rm Poly}_k$ stands symbolically for a polynomial of degree $k$ of $x$.
Computation shows the following.

\begin{prp}         
$$  \begin{array}{lll}
P_{a-} &=& (x-1)^2 \partial^2  -(x-1)(3a+3b+4c+2g+2-(a+c)/x) \partial \\
  && + (2a+2b+2c+g+2)(a+b+2c+g+1-c/x),\\
    P_{c-}&=& \partial^2-(xa+xb+2xc-a-c)/x/(x-1)\,\partial\\
    && \quad +c(2a+2+2b+2c+g)/((x-1)x), \\
  P_{g-}&=& cp_2\partial ^2+cp_1/(x(x-1))\partial  + cp_0/(x(x-1)), \\
\end{array} $$
where
$$  \begin{array}{lcl}
cp_2 &=& (2c+g)(a+b+g)x^2-(2b+g)(2c+g)x+(2b+g)(a+c+g), \\
cp_1 &=& -(2c+g)(a+b+g)(3a+2g+4c+3b+2)x^3 \\
&&\quad + (2c+g)((a-b)(2+3a -2b +3c)\\
&&\quad + (2b +g)(3+7a +2b+6c +3g))x^2\\
&&\quad  -(2b+g)((a - c) (a + b + 2 c) + (2 c + g) (1 + 3 a + 2 b + 4 c + g))x\\
&& \quad +(a+c)(a+c+g)(2b+g), \\
cp_0 &=& (2a+2+2b+2c+g)\ [(2c+g)(a+b+g)(a+2c+b+1+g)x^2 \\
  &&\quad -(2c+g)((a - b) (1 + a - b + c) - (2 b + g) (1 + 2 a + 2 c + g))x\\
  &&\quad  + c(2b+g)(a+c+g)].
\end{array} $$
\end{prp}

Further study of these shift operators is made in \cite{EOY}.

\subsection{S-values}\label{S-values}
By use of the explicit forms of shift operators, we can see
  the following.

\begin{thm}The S-values are given as follows
\begin{eqnarray*}
Sv_{a-}&=& P_{a+}(a-1)\circ P_{a-} \equiv a(a+\frac{g}2)M,\\
Sv_{c-}&=& P_{c+}(c-1)\circ P_{c-} \equiv c(c+\frac{g}2)M,\\
Sv_{g-}&=&
P_{g+}(g-2)\circ P_{g-} \equiv (\frac{g}2-\frac12)(a+\frac{g}2)(b+\frac{g}2)(c+\frac{g}2)N,
\end{eqnarray*}
mod $S\!E_3$, up to constant multiplication,

$$\begin{array}{lll}\displaystyle M=(\delta+g+1)(\delta+\frac{g}2+1),
  \quad \displaystyle N=(\delta+g)(\delta+g+1)(\delta+\frac{g}2+1).
  \end{array}$$
and $\delta:=a+b+c$. Here, for example, $P_{a+}(a-1)\circ P_{a-}$ is an abbreviation of  $P_{a+}(a\to a-1)\circ P_{a-}(a)$. Note 
$$P_{a-}(a+1)\circ P_{a+}=(P_{a+}(a-1)\circ P_{a-})(a\to a+1),$$
$$P_{g-}(g+2)\circ P_{g+}=(P_{g+}(g-2)\circ P_{g-})(g\to g+2).$$
\end{thm}
Let $\Delta$ denote one of 
$$\begin{array}{l}-a,\ -b,\ -c,\ a+g/2+1,\ b+g/2+1,\ c+g/2+1,\\[2mm]
 -(\delta+g/2+1),\ \delta+g+2,\qquad 1/2-g/2.
  \end{array}$$
\begin{thm}\label{Sv_SE3}If $\Delta=0$ or $1$, then an S-value vanishes. \red{If $\delta+g=0$ or $\delta+g+3=0$, then an S-value also vanishes.}
\end{thm}
Proposition \ref{Sred} and Theorem \ref{red_atoap} lead to
\begin{cor}\label{red_cond_SE3}$(\cite{Mim})$ If $\Delta\in\mathbb{Z}$, then $S\!E_3(a,b,c,g)$ is reducible.
\end{cor}

\subsection{Reducible cases of $S\!E_3$}\label{red_cases}

\subsubsection{Example: Factorization of the shift relations of $S\!E_3$  when $a=-1$, $0$}\label{factor_shift_rel_SE3}
We study factorizations of $S\!E_3$  when $a=-1,0$ by observing
the shift relations. For the other cases, the situation is analogous.

In this subsection,
for notational simplicity we write $E$ in place of $S\!E_3$.
  Writing $E(-1)$ for $S\!E_3(a=-1)$, and $P_{a-}(0)$ for $P_{a-}(a=0)$, etc, we have shift relations
$$(1):\ E(-1)\circ P_{a-}(0)=Q_{a-}(0)\circ E(0),\quad (2):\ E(0)\circ P_{a+}(-1)=Q_{a+}(-1)\circ E(-1),$$
which factorize respectively as
$$(1):\ [1,2]\ [1,1]=[1,1]\ [2,1],\qquad (2):\ [2,1]\ [2]=[2]\ [1,2].$$
In fact, the equations and the shift operators factor as
\[\begin{array}{ll}E(-1)=En_1\circ En_2,\ &P_{a-}(0)=P_1\circ P_2,\\[1mm]
&Q_{a-}(0)=Q_1\circ Q_2,\ E(0)=E0_1\circ E0_2,\end{array}\]
where 
\[\begin{array}{ll}\def\arraystretch{1.4}\setlength\arraycolsep{2pt}
En_1&= (x - 1) x \partial - 2 b x - 2 c x - g x - 1,\\
En_2&=x (x - 1) \partial^2 - (b x + 4 c x + g x - 3 c - g - x + 1) \partial + (2 b x + 4 c x + 2 g x - 2 c - g) c/x,\\
P_1&=(x - 1) \partial - (b x + 2 c x + g x - c + x)/x,\\
P_2&=(x - 1) \partial - (2 b + 2 c + g + 2),\\
Q_1&=En_1,\\
Q_2&=\partial -( b x + 2 c x + x \partial + g x - c + x - 1)/x^2,\\
E0_1&=x^2 (x - 1) \partial^2 + (-x^2 b - 4 x^2 c - x^2 g + 3 c x + g x) \partial + (2 b x + 4 c x + 2 g x - 2 c - g + 2 x - 1) c,\\
E0_2&=P_2.\end{array}\]

\begin{dfn}\label{essentially}
   We say two operators $R_1$ and $R_2$ are
   {\it essentially the same} if they are related as
   $$R_1=x^{\mu_0}(x-1)^{\mu_1}\ R_2\circ x^{\nu_0}(x-1)^{\nu_1},$$
together with renaming the exponents.
We write $R_1\sim R_2$. 
  \end{dfn}

For the present case, we have
$$E0_1\sim Q_{a+}(-1),\quad E0_2\sim En_1,\quad P_{a+}(-1)\sim En_2,$$
while the relation (2) turns out to be a trivial identity. \par
Since 
$En_1= Q_1$ and $P_2= E0_2,$
canceling them from both sides of (1), we get a relation 
\begin{equation}\label{12_21} En_2\circ P_1=Q_2\circ E0_1\end{equation}
of type $[2]\ [1]=[1]\ [2].$

Since We have
$$x^{-c}En_2\circ x^c={}_2E_1(\alpha+1,\beta,\gamma+1),$$
$$x^{-c-1}E0_1\circ x^c={}_2E_1(\alpha,\beta,\gamma),$$
where $$\alpha=-b-c-g-1,\ \beta=-c,\ \gamma=-c-g,$$ 
the operators 
$$x^{-c}P_1\circ x^c=(x-1)\partial -b - c - g - 1,$$
$$x^{-c}Q_2\circ x^{c+1}=(x-1)\partial -g - c - b$$ 
are shift operators of the Gauss equation for the shift $(\alpha,\gamma)\to(\alpha+1,\gamma+1)$, and (\ref{12_21}) is the shift relation.

\subsubsection{Reducible types}

As in \S \ref{S-values}, let $\Delta$ be one of 
$$-a,\dots,\ a+g/2+1,\dots,\ -(\delta+g/2+1),\ \delta+g+2,\ 1/2-g/2$$
  and assume it is an integer. The previous examples and
  Proposition 2.11 lead to
\begin{prp}\label{factorization_1}
Suppose  $\{a,b,c,g\}$ are generic but $\Delta\in\mathbb{Z}$.  If $\Delta$ is a non-positive integer then $S\!E_3(a,b,c,g)$ is reducible of type $[2,1]$, and if $\Delta$ is a positive integer then $S\!E_3(a,b,c,g)$ is reducible of type $[1,2]$: 
$$\begin{array}{cccccccc} 
\Delta= &\dots&-1  &0   &|&1    &2 &\dots\\[2mm]
{\rm type}: &\dots&[2,1]&[2,1]A0&&[1,2]A0&[1,2]&\dots
  \end{array}$$
where $A0$ means the factors have no apparent singularities (cf. Notation \ref{nttn0}). 
\end{prp}


\section{Review of fundamental properties of $E_6$}\label{E6}

We recall some of fundamental properties of $E_6$ from \cite[\S 1]{HOSY1}.

The Fuchsian differential equation 
$$E_6=x^3(x-1)^3\partial^6+\cdots+p_0,$$
where $p_0$ is the constant term $s(s+1)(s+2)e_7e_8e_9$, 
with the Riemann scheme
\footnote{Recall that we {\it always assume} that local solutions corresponding to local exponents with integral difference, such as $0$, $1$, $2$; $s$, $s+1$, $s+2$,  has no logarithmic term, and the other exponents, such as $e_1$, $e_2$, $\dots$,  are generic.}
$$R_6:\left(\begin{array}{lcccccc}x=0:&0&1&2&e_1&e_2&e_3\\x=1:&0&1&2&e_4&e_5&e_6\\x=\infty:&s&s+1&s+2&e_7&e_8&e_9\end{array}\right),\quad e_1+\cdots+e_9+3s=6,
$$
 has the following $(\theta,\partial)$-form 
  \begin{eqnarray}\label{eqT}T=T_0(\theta )+T_1(\theta )\partial+T_2(\theta )\partial^2+T_3(\theta )\partial^3,\quad \theta=x\partial\end{eqnarray}
where
{\setlength\arraycolsep{2pt} \def\arraystretch{1.1}
  \begin{eqnarray} 
\quad T_0&=&(\theta +2+s)(\theta +1+s)(\theta +s)B_0,\quad B_0=(\theta +e_7)(\theta +e_8)(\theta +e_9), \label{eqT0}\\
\quad      T_1&=&(\theta +2+s)(\theta +1+s)B_1,\quad B_1=T_{13}\theta ^3+T_{12}\theta ^2+T_{11}\theta +T_{10}, \label{eqT1}\\
\quad    T_2&=&(\theta +2+s)B_2,\quad B_2=T_{23}\theta ^3+T_{22}\theta ^2+T_{21}\theta +T_{20}, \label{eqT2}\\
\quad    T_3&=&(-\theta -3+e_1)(-\theta -3+e_2)(-\theta -3+e_3), \label{eqT3}
  \end{eqnarray}
  }
and
\[\def\arraystretch{1.2}\setlength\arraycolsep{3pt} \begin{array}{rcl}
  T_{13}&=&-3,\quad T_{23}=3,\quad  T_{12}=-9+s_{11}-2s_{13},\quad  T_{22}=18+s_{13}-2s_{11},\\
T_{11}&=&-8+(s_{11}^2+2s_{11}s_{13}-s_{12}^2+s_{13}^2)/3+s_{11}-5s_{13}-s_{21}+s_{22}-2s_{23},\\
T_{21}&=&35+(-s_{11}^2-2s_{11}s_{13}+s_{12}^2-s_{13}^2)/3-7s_{11}+5s_{13}+2s_{21}-s_{22}+s_{23},\\
T_{20}&=&-T_{10}+19+(s_{11}^2s_{13}-s_{11}s_{12}^2+s_{11}s_{13}^2-s_{12}^2s_{13})/9+(s_{13}^3+s_{11}^3-2s_{12}^3)/27\\
&&+(-2s_{11}^2-4s_{11}s_{13}+s_{11}s_{22}+2s_{12}^2+s_{22}s_{12}-2s_{13}^2+s_{22}s_{13})/3\\
&&
-5s_{11}+4s_{13}+3s_{21}-2s_{22}-s_{31}-s_{32}-s_{33},\\
2T_{10}&=&(2s_{11}+s_{11}^2+2s_{11}s_{13}-s_{12}^2
+s_{13}^2-s_{11}s_{21}+2s_{11}s_{23}+s_{12}s_{22}\\
&& -2s_{13}s_{21}+2s_{13}s_{22}+s_{13}s_{23}-14s_{13})/3+2(s_{11}^2s_{13}-s_{12}^2s_{13})/9
\\
  &&+2(s_{11}^3-s_{12}^3)/27 -5-s_{21}+s_{31}+s_{22}-s_{32}-5s_{23}-3s_{33}.
\end{array}\]
Here $s_*$ are symmetric polynomials of the exponents:
\def\arraystretch{1.1}\setlength\arraycolsep{1pt}
\begin{equation}\label{e19tos}\begin{array}{rl}      
s_{11}&=e_1+e_2+e_3,\quad s_{12}=e_4+e_5+e_6,\quad s_{13}=e_7+e_8+e_9,\\
s_{21}&=e_1e_2+e_1e_3+e_2e_3,\quad s_{22}=e_4e_5+e_4e_6+e_5e_6,\\ 
s_{23}&=e_7e_8+e_7e_9+e_8e_9,\quad s_{31}=e_1e_2e_3,\quad s_{32}=e_4e_5e_6,\\ 
s_{33}& = e_7e_8e_9,\quad s=-(s_{11}+s_{12}+s_{13}-6)/3.
\end{array}\end{equation}

\subsection{Shift operators of $E_6$}
We quote some results of \cite{HOSY1}. Set
\[ \def\arraystretch{1.1} \begin{array}{l}
  {\bm e}_1=(e_1,e_2,e_3),\ {\bm e}_4=(e_4,e_5,e_6),\ {\bm e}_7=(e_7,e_8,e_9),\quad {\bm e}=({\bm e}_1,{\bm e}_4,{\bm e}_7),\\
  {\bm e}_1\pm{\bf 1}=(e_1\pm1,e_2\pm1,e_3\pm1),\ {\bm e}_1+\ell={\bm e}_1+\ell{\bf 1},\dots
\end{array}\]

\begin{thm}\label{shiftopE6} {\rm (\cite[\S 6]{HOSY1})}\
  Equation $E_6$ has shift operators relative to
the shifts of blocks ${\bm e}_i\to{\bm e}_i\pm{\bf 1}\ (i=1, 4, 7).$ 
\end{thm}

\begin{dfn}\label{E6PQnotation}
  The shift operator for the shift ${\bm e}_1\to{\bm e}_1\pm{\bf 1}$
  is denoted by $P_{\pm00}$,
  for the shift ${\bm e}\to ({\bm e}_1+{\bm 1},{\bm e}_4-{\bm1},{\bm e}_7) $
  by $P_{+-0}$, and so on.
  More precisely, if $P$ and $Q$ solve the equation
  $$E_6({\bm e}_1+\epsilon_1{\bf 1},{\bm e}_4+\epsilon_4{\bf 1},
  {\bm e}_7+\epsilon_7{\bf 1})\circ P
  =Q\circ E_6({\bm e}),\quad \epsilon_1,\epsilon_4,\epsilon_7,=-1,0,1,$$
they are called $P_{\epsilon_1\epsilon_4\epsilon_7}$ and $Q_{\epsilon_1\epsilon_4\epsilon_7}$, respectively. The solution $(P,Q)$ is unique up to multiplicative constant.
\end{dfn}

For example, 
$$ P_{-00}=(x-1)\partial-r,  \quad P_{0-0}:=x\partial-r, \qquad P_{--+}=\partial.$$

\subsection{S-values and reducibility conditions of $E_6$}\label{E6S}
The S-values are defined in \S 2.6.3. 
The simplest is $Sv_{--+}$ given by $P_{--+}=\partial$ and its inverse $P_{++-}$ as
$$\begin{array}{ll}Sv_{--+}&=\
  P_{++-}({\bm e}_1-{\bm1},{\bm e}_4-{\bm1},{\bm e}_7+{\bm1})\circ P_{--+}\\
&=\  E_6-p_0 \equiv -p_0=r(r-1)(r-2)e_7e_8e_9\quad{\rm mod}\ E_6.\end{array}$$

\begin{prp}\label{S-valueE6} {\rm (\cite[\S 6.1]{HOSY1})}
  The three S-values of the above three shift operators:
\[\def\arraystretch{1.3}\setlength\arraycolsep{3pt} \begin{array}{rcl}
Sv_{--+} &=& P_{++-}({\bm e}_1-{\bf 1}, {\bm e}_4-1, {\bm e}_7+1)\circ P_{--+}
       =-s(s+1)(s+2)e_7e_8e_9,\\
Sv_{-00}&=& P_{+00}({\bm e}_1-{\bf 1})\circ P_{-00}
      =-s(s + 1)(s + 2)(s + e_4)(s + e_5)(s + e_6),\\
%
Sv_{0-0} &=& P_{0+0}({\bm e}_4-{\bf 1})\circ P_{0-0}
      =s(s + 1)(s + 2)(s + e_1)(s + e_2)(s + e_3).\\
\end{array}\]
\end{prp}

Recall that the equation is reducible when the S-value of a shift operator
vanishes (Proposition \ref{Sred}),
and that if an equation is connected by a shift operator
with a reducible equation, it is also reducible
(Theorem \ref{red_atoap}). 
So the S-values above and the shift operators
(Theorem \ref{shiftopE6}) lead  to

\begin{thm}\label{redcondE6}If one of
$$s,\quad e_i+s\ (i=1,\dots,6),\quad e_7,\ e_8, e_9$$
is an integer, then the equation $E_6$ is reducible.
\end{thm}

\subsection{Reducible cases of $E_6$}\label{E6Red}

Types of factorization of reducible $E_6$ are summarized.

\subsubsection{$e_9\in\mathbb{Z}$. When $e_9=0$ and $1$, $E_5$ appears}\label{E6e90}

\begin{prp}\label{E6red_e9}
  If $e_9\in {\bf Z}$, then $E_6$ factorizes as follows: when $e_9$ is
  a non-positive integer, the type of factorization is $[51]$ and, when
  it is a positive integer, $[15]:$
   $$\begin{array}{ccccccccc}
e_9=\ & \cdots &-2    &-1   & 0   & 1     & 2    &3      &\cdots\\   
\ &\  \ &\ [51]\ &\ [51]\ &\ [51]A0\ &\ [15]A0\ &\ [15]\ &\ [15]\ &\     
  \end{array}$$
  The notation $A0$ means that the operators have no singularity other than
  $\{0, 1, \infty\}$ (cf. Notation \ref{nttn0}).
    In particular, when $e_9=0,1$, the factor $[5]$ is essentially $E_5$.
\end{prp}

\begin{remark}\label{factorSameType}The table in the Proposition means
  when $e_9=2,3,\dots,$ the equation $E_6$ factors of the same type [15], and when $e_9=-1,-2,\dots,$ the equation $E_6$ factors of the same type [51]; this is seen by Proposition \ref{FactorType}. 
In this article, when an equation is reducible, we always give this kind of table without repeating this explanation.
\end{remark}

\subsubsection{$e_1+s\in\mathbb{Z}$}\label{E6e1rr}
 $$\begin{array}{ccccccccc}
e_1+s=\ \ &\  \cdots \ &\ -2    \ &\ -1   \ &\  0   \ &\  1     \ &\  2    \ &\ 3      \ &\ \cdots\\   
\ &\  \ &\ [51]\ &\ [51]\ &\ [51]A0\ &\ [15]A0\ &\ [15]\ &\ [15]\ &\     
 \end{array}$$
By the change $x\to1/x$, this case is converted
to the case $e_9\in\mathbb{Z}$.
So when $e_1+s=0,1$, the factors $[5]$ are essentially $E_5$. For others, the factor $[5]$ has singularities out of $\{0,1,\infty\}$.

\subsubsection{$s\in\mathbb{Z}$. When $s=-2,-1,0,1$, $E_3$ appears}
\label{E6Reds}

\begin{prp}\label{prp1113}
  If $s\in\mathbb{Z}$, $E_6$ is reducible of type $\{3111\}$:
  $$\begin{array}{ccccccccc}
      s=\ \ &\  \cdots \ &\ -3 \ &\ -2  \ &\  -1  \ &\  0      \ &\  1      \ &\ 2       \ &\ \cdots\\   
          \ &\         \ &\ [3111]\ &\ [3111]A0\ &\ [1311]A0\ &\ [1131]A0\ &\ [1113]A0\ &\ [1113]\ &\ 
  \end{array}$$
  When $s=-2,-1,0,1$, the factor $[3]$ is essentially $E_3$,
  and no apparent singularities appear.
\end{prp}


\section{$S\!E_6$: a specialization of $E_6$}\label{SE6}

The equation $S\!E_6$ is, by definition, a codimension-2 specialization of the exponents of $E_6$:
\begin{equation}e_3-2e_2+e_1=e_4-2e_5+e_6=e_7-2e_8+e_9.\label{E6restEar}
\end{equation}
This equation has many nice properties  such as integral representation of solutions and shift operators that do not come from those of $E_6$.

\subsection{Definition of $S\!E_6$ as a middle convolution of $S\!E_3$}
\label{SE6MC}

As is stated in Section 3, any solution of the equation $S\!E_3=(x-1)^2 x^2\partial^3+\cdots$ admits the integral representation:
$$\int\omega(x),\quad \omega(x):=\prod_{i=1,2} t_i^a(t_i-1)^b(t_i-x)^c\cdot (t_1-t_2)^g\ dt_1\wedge dt_2.$$
To a solution $v(x)=\int\omega$ of $S\!E_3$, we multiply $x^{p}(x-1)^{q}$:
$$v(x)\ \to\  v_{pq}(x)=x^{p}(x-1)^{q}v(x),$$
and apply the Riemann-Liouville transformation (see Definition 2.1)
with parameter $\mu=r+1$:
$$v_{pq}(x)\ \to\ \int v_{pq}(t)(t-x)^{r}\, dt,$$
and we get a function of $x$ defined by the integral
$$\int \prod_{i=1,2} t_i^a(t_i-1)^b(t_i-t_3)^c\cdot (t_1-t_2)^g\cdot t_3^{p}(t_3-1)^{q}(t_3-x)^{r}\ dt_1\wedge dt_2\wedge dt_3.$$
Note that the integrand is invariant under the change $t_l\leftrightarrow t_2$. It is well-known (cf. \cite{AK}) that a function defined by such an integral satisfy a Fuchsian differential equation whose order is the number of bounded chambers in the real $(t_1,t_2,t_3)$-space cut out by the planes defined by the factors of the integrand.

\subsubsection{Number of bounded chambers and the order of the equation}\label{SE6MCbdd}

Consider in the real $(t_1,t_2,t_3)$-space the arrangement of 
ten planes:
$$ \prod_{i=1,2,3}t_i(t_i-1)\cdot \prod_{1\le i<j\le3}(t_i-t_j)\cdot(t_3-x)=0.$$
The number of bounded chambers cut out by this arrangement is twelve, in fact,
for simplicity we assume $x>1$, then we find eight tetrahedra:
$$\begin{array}{lll}
&0<t_i<t_j<t_k<1\qquad &\{i,j,k\}=\{1,2,3\},\\ [1mm]
&1<t_i<t_j<t_3<x &\{i,j\}=\{1,2\},\end{array}$$
and four triangular prisms:
$$\begin{array}{llll}
&0<t_i<t_j<1,\quad &1<t_3<x\qquad &\{i,j\}=\{1,2\},\\ [1mm]
  &1<t_i<t_3<x,\quad &0<t_j<1\qquad &\{i,j\}=\{1,2\}.\end{array}$$
Though there are twelve bounded chambers, the number of bounded chambers modulo the change $t_1\leftrightarrow t_2$ is six. So the function defined by the integral satisfy an equation of order six; let us call this equation $E\!ar=E\!ar(a,b,c,g,p,q,r)$ for the moment. It is not practical to find $E\!ar$ by differentiating the integral. We make use of a middle convolution.

\subsubsection{$E\!ar$ is $S\!E_6$}\label{SE6MCEar}
When $v$ is a solution of $S\!E_3=x^2(x-1)^2\partial^3+\cdots$, the equation satisfied by $v_{pq}=x^{p}(x-1)^{q}v(x)$ can be written in $(x,\partial)$-form as
  $$x^3(x-1)^3\partial^3+\cdots$$
Since this operator has a term $x^6\partial^3$, to get a $(\theta ,\partial)$-form, we have to multiply $\partial^3$, 
of order at least 3.
After multiplying $\partial^3$ from the left, 
we can write the equation in $(\theta ,\partial)$-form, which is a linear combination of 
$\theta ^i\circ\partial^j.$
Apply the middle convolution (\S \ref{GenAdd})
with parameter $\mu=r+1$, that is,
replace $\theta $ by $\theta -r-1$.
 We then eventually get an operator
 $$E\!ar=x^3(x-1)^3\partial^6+\cdots,$$
and find

\begin{thm}\label{etoar} The equation $E\!ar$ turns out to be
  $S\!E_6$, a specialization
  of the equation $E_6$:
  \[\def\arraystretch{1.1}\setlength\arraycolsep{2pt}  
  \begin{array}{lll}
e_1=p+r+1, &e_2=a+c+p+r+2, &e_3=2a+2c+g+ p+r+3,\\
e_4=q+r+1, &e_5=b+c+q+r+2, &e_6=2b+2c+g+ q+r+3,\\
e_7=-2c-pqr-1, &e_8=-\delta-c-pqrg-2, &e_9=-2\delta-pqrg-3,\end{array}  \]
  \setlength\arraycolsep{3pt} 
  where $pqr=p+q+r,$ $pqrg=pqr+g,$ 
  $\delta=a+b+c$.
\end{thm}

Note that
\begin{equation}e_3-2e_2+e_1=e_4-2e_5+e_6=e_7-2e_8+e_9 \quad (=g).
\end{equation}
  Its Riemann scheme is $R_6$ with $s=-r$
$$\left(\begin{array}{ccclll}
  0\ &\  1\ &\  2\ &\  e_1 \ &\ e_2 \ &\ e_3\\
  0\ &\  1\ &\  2\ &\  e_4 \ &\ e_5 \ &\ e_6\\
    -r\ &\  1-r\ &\  2-r\ &\  e_7\ &\  e_8\ &\  e_9
  \end{array}\right).$$
Note also that this procedure of getting $S\!E_6$ from $S\!E_3$ is quite parallel to that of getting $E_6$ from $E_3$ presented in \cite{HOSY1} \S 3.1.
 
\subsection{Shift operators of $S\!E_6$}\label{SE6Shift}
The equation $E=S\!E_6$ has 7 parameters $a,b,c,g,p,q,r$.
Let the letter $e$ denote one of these parameters. We show that the equation
\begin{equation} \label{epqe}
(EPQE):\quad E(e\pm1)\circ P = Q\circ E(e),
\end{equation}
($e\pm 2$ when $e=g$), has always nontrivial solution $(P,Q)$:

\begin{thm}\label{shiftopSE6}The equation $S\!E_6$ admits shift operators for the shifts $a\to a\pm1$, $b\to b\pm1$, $c\to c\pm1$, $g\to g\pm2$,
  $p\to p\pm1$, $q\to q\pm1$, and $r\to r\pm1$.
  \end{thm}
Recall the notation $P_{e\pm}$ and $Q_{e\pm}$ as was
explained in \S \ref{se3op}.
Shift operators for $r\rightarrow r\pm 1$,
$p\rightarrow p\pm 1$ and  $q\rightarrow q\pm 1$ come from those of $E_6$. 
By the expression of $e_1,\, \dots,\, e_9$ in terms of $a,\, \dots,\, r$,
it is easy to see that $r\rightarrow r\pm 1$  corresponds to
${\bm e}\to ({\bm e}_1\pm1, {\bm e}_4\pm1, {\bm e}_7\mp1).$ Thus, we have 
$$P_{r-} = P_{--+}, \quad P_{r+} = P_{++-},$$
and the same for the $Q$'s;
see Definition \ref{E6PQnotation} for the right hand-side. Similarly, we have:
$$P_{p-} = P_{-0+}, \quad P_{p+} \ = \  P_{+0-}, \quad P_{q-} = P_{0-+}, \quad P_{q+} \ =\  P_{0+-}.$$
The S-values are given as
$$Sv_{r-} = P_{r+}(r-1)\circ P_{r-},\quad Sv_{p-} = P_{p+}(p-1)\circ P_{p-},
  \quad Sv_{q-}= P_{q+}(q-1)\circ P_{q-},$$
  whose explicit values are given in \S \ref{E6S}.  

\subsection{From the shift operators of $S\!E_3$ to those of $S\!E_6$}\label{shiftop_SE3_to_SE6}

Since $S\!E_6$ is connected by addition and middle convolution with $S\!E_3$, the shift operators for the shifts of $(a,b,c,g)$   can be derived  from the corresponding shift operators of $S\!E_3$,
given in \S \ref{se3op}, 
by following the recipe in \cite{Osh} Chapter 11.  

\subsubsection{From $S\!E_6$ to $S\!E_3$}

While $S\!E_6$ was defined by use of $S\!E_3$ in \S \ref{SE6MC},
we can get the expression $S\!E_3$ conversely by use of $S\!E_6$ as follows:
  \begin{enumerate}
  \item In the $(\theta,\partial)$-form of $S\!E_6$, replace $\theta$ by $\theta+r+1$. The resulting equation is divisible from the left by $\partial^3$.
  \item Divide by $\partial^3$.
  \item The local exponents at 0 and 1 of the resulting equation are
    $$-r+e_1-1,\ -r+e_2-1,\ -r+e_3-1;\quad -r+e_4-1,\ -r+e_5-1,\ -r+e_6-1,$$
where $-r+e_1-1=p,-r+e_4-1=q$ by Theorem \ref{etoar}.
Multiply $X^{-1}$ from the left and $X$ from the right, where
$ X:=x^p(x-1)^q$.
Then we get $S\!E_3$ up to a simple modification as follows.\end{enumerate}
\begin{prp}
  \begin{equation}\label{SSE3_SE6}
S\!E_3=x^{-1}(x-1)^{-1}X^{-1}\cdot\partial^{-3}\cdot\partial^{r+1}\cdot S\!E_6\cdot \partial^{-r-1}\cdot X,\quad X=x^p(x-1)^q.  \end{equation}
\end{prp}
In other words, given $u_3$ such that $S\!E_3u_3=0$, define $u_6$ as
$$u_6:=\partial^{-r-1}Xu_3;$$
then $S\!E_6u_6=0$, and conversely given $u_6$ such that $S\!E_6u_6=0$, define $u_3$ as
$$u_3:=X^{-1}\partial^{r+1}u_6;$$
then $S\!E_3u_3=0$.

\subsubsection{Recipe for computing shift operators}\label{recipe_3_to_6}
For a shift $sh:(\mathbf{a})\to sh(\mathbf{a})$ of
$\mathbf{a}=(a,b,c,g)$, let $S\!E_3(sh(\mathbf{a}))$
and $S\!E_6(sh(\mathbf{a}))$ be the shifted operators of
$S\!E_3(\mathbf{a})$ and $S\!E_6(\mathbf{a})$, respectively.

\begin{itemize}
\item 
Given $v_3$ such that $S\!E_3(sh(\mathbf{a}))v_3=0$, define $v_6$ as
$$v_6:=\partial^{-r-1}Xv_3;$$
then $S\!E_6(sh(\mathbf{a}))v_6=0$, and conversely given $v_6$
such that $S\!E_6(sh(\mathbf{a}))v_6=0$, define $v_3$ as
$$v_3:=X^{-1}\partial^{r+1}v_6;$$
then $S\!E_3(sh(\mathbf{a}))v_3=0$.

\item Let $P=P_3$ be the shift operator for the shift $sh$ of $S\!E_3$
$$P:{\rm Sol}(S\!E_3(\mathbf{a}))\to{\rm Sol}(S\!E_3(sh(\mathbf{a}))).$$ 
When $u_3$ solves $S\!E_3((\mathbf{a}))u_3=0$, 
$$v_3:=Pu_3\quad{\rm solves}\quad S\!E_3(sh(\mathbf{a}))v_3=0,$$ 
turning to $E_6$, 
$$v_6:= \partial^{-r-1}\cdot Xv_3\quad{\rm solves}\quad S\!E_6(sh(\mathbf{a}))v_6=0.$$ 
Rewriting $v_3=Pu_3$, we have
$$X^{-1}\cdot\partial^{r+1}v_6=P\cdot X^{-1}\cdot\partial^{r+1}u_6. $$
That is
\begin{equation}\label{mc}
  \partial^{r+1}v_6=X\cdot P\cdot X^{-1}\cdot\partial^{r+1}u_6.
\end{equation}

\item To apply the middle convolution formula:
  $$\partial^{-r}\circ \theta\circ\partial^{r}=\theta -r,$$
we proceed as follows:
\begin{itemize}
\item if $XPX^{-1}$ has poles (at $x=0,1$), multiply a power of $x$ and a power of $x-1$ from the left to both sides of \eqref{mc} to kill the  poles,
\item multiply from the left some powers of $\partial$ so that both sides can have $(\theta,\partial)$-forms,
\item replace $\theta$ by $\theta-r-1,$ and name the left-hand side operator as $R_2$ and the right one as $R_1$.
\end{itemize}
In this way, \eqref{mc} is transformed to 
$$R_2\ v_6 = R_1 \ u_6.$$

\item
  Apply Euclidean algorithm to
  $\mathbf{E}=S\!E_6(sh(\mathbf{a}))$ and $R_2$ ($\nu={\rm ord}R_2$). 
It goes :
  $$\begin{array}{lll}
  \mathbf{E}=q_1R_2+r_1,\quad&{\rm ord}(r_1)=\nu-1,
  \quad&r_1=x_1\mathbf{E}+y_1R_2,\quad y_1=-q_1,\\[2mm]
  R_2=q_2r_1+r_2,&{\rm ord}(r_2)=\nu-2,&r_2=x_2\mathbf{E}+y_2R_2,
  \quad y_2=1+q_2q_1,\\[2mm]
  r_1=q_3r_2+r_3,&{\rm ord}(r_3)=\nu-3,&r_3=x_3\mathbf{E}+y_3R_2,
  \quad y_3=y_1-q_3y_2,\\
\cdots\cdots&\cdots&\cdots\\
r_{\nu''}=q_\nu r_{\nu'}+r_\nu,&{\rm ord}(r_\nu)=0,
&r_\nu=x_\nu \mathbf{E}+y_\nu R_2,\quad y_\nu=y_{\nu''}-q_\nu y_{\nu'},
\end{array}$$
where $\nu'=\nu-1, \nu''=\nu-2$. Putting
$$B:=y_\nu/r_\nu,$$
since $\mathbf{E}v_6=0$, we have $$ v_6=BR_2\ v_6=BR_1\ u_6.$$ 
Thus we have the shift operator $P_6$ of $S\!E_6$ for the shift $sh$:
  $$P_6=BR_1\quad {\rm modulo}\quad S\!E_6(\mathbf{a}).$$
\end{itemize}

\subsubsection{Example $sh:a\to a+1$, $P_6=P_{a+}$}

\begin{itemize}
\item
  Since $X PX^{-1}=x^2(x-1)^2\partial^2+\cdots$
  where $P=P_{a+}$ for $S\!E_3$, multiply $\partial^2$ from the left and get $(\theta,\partial)$-form of $\partial^2X P X^{-1}.$ Thus we have
$$\partial^2\ v_6=\partial^{-r-1}\cdot\partial^2XPX^{-1}\cdot\partial^{r+1}\ u_6.$$
Substitute $\theta \to \theta-r-1$
in the $(\theta,\partial)$-form of the right-hand  side and we have
$$R_2v_6=R_1u_6,\quad R_2=\partial^2,\quad R_1=subs(\theta=\theta-r-1,\partial^2XPX^{-1}).$$
\item
Apply Euclidean algorithm to $\mathbf{E}=S\!E_6(a+1)$ and $R_2$ $(\nu={\rm ord}R_2=2.)$

It goes:
  $$\begin{array}{llll}
    \mathbf{E}&=q_1R_2+r_1,\quad&{\rm ord}(q_1)=4,{\rm ord}(r_1)=1,\quad&r_1=\mathbf{E}-q_1R_2,\\[2mm]
    R_2&=q_2r_1+r_2,&{\rm ord}(q_2)=1,{\rm ord}(r_2)=0,&r_2=-q_2\mathbf{E}+(1+q_2q_1)R_2,\end{array}$$
$$ 1=-r_2^{-1}q_2\mathbf{E}+BR_2,\quad {\rm where}\quad B:=r_2^{-1}(1+q_2q_1).$$
and so
$$v_6=BR_2\ v_6=BR_1\ u_6.$$ 
Thus we have the shift operator $P_6$ of $S\!E_6$ for the shift $a\to a+1$:
  $$P_6=BR_1\quad {\rm modulo}\quad S\!E_6(\mathbf{a})$$
$P_6=(x^3(x-1)^3\partial^5+\cdots)/C$,
where $C$ is a polynomial of $\mathbf{a}$.  
\end{itemize}

\subsubsection{Example $sh:a\to a-1$, $P_6=P_{a-}$}

\begin{itemize}
\item 
  Since $X PX^{-1}=(x-1)^2\partial^2+\cdots$
  where $P=P_{a-}$ for $S\!E_3$, multiply $\partial^2x^2$ from the left and get $(\theta,\partial)$-form of $\partial^2x^2X P X^{-1},$  and $\partial^2x^2=(\theta+1)(\theta+2)$.
Thus we have
$$\partial^{-r-1}(\theta+1)(\theta+2)\partial^{r+1}\ v_6=\partial^{-r-1}\cdot\partial^2x^2XPX^{-1}\cdot\partial^{r+1}\ u_6.$$
Substitute $\theta \to \theta-r-1,$
in the $(\theta,\partial)$-form of both sides and we have
$$R_2v_6=R_1u_6,\quad R_2=(\theta-r)(\theta-r+1),\quad R_1=subs(\theta=\theta-r-1,\partial^2x^2XPX^{-1}).$$

\item
Apply Euclidean algorithm to $\mathbf{E}=S\!E_6(a-1)$ and $R_2\ ({\rm ord}(R_2)=2)$. 
It goes exactly the same to the case of the previous subsubsection, and we get
the shift operator $P_6$ of $S\!E_6$ for the shift $a\to a-1$:
$$P_6=(x^3(x-1)^3\partial^5+\cdots)/C,$$ where $C$ is a polynomial of $\mathbf{a}$. 
\end{itemize}

\begin{remark} Theoretically the above recipe works well, however the 
operators appearing in the process and the resulting ones are extremely complicated.
Anyway since this recipe guarantees the existence of the shift operators,  we directly solve the shift equation $(EPQE)$ to get a shorter expression
in the next subsection.
\end{remark}

\subsection{Solving the shift equation $(EPQE)$} \label{secshift}

  Shift operators for $a\to a\pm1$ will be found in \S \ref{SE6Shifta}
  and \S \ref{SE6Shiftan}. 
For $b\rightarrow b\pm 1$, we can make use of the symmetry
\[(a, b, p, q, x, \partial)\longrightarrow (b, a, q, p, 1-x, -\partial).\]

For $c\to c\pm1$, we make a detour
$$(c, p, q)\underset{P_{q\pm}}\to(c, p, q\pm1)\underset{P_{p\pm}(q\pm1)}
\to (c, p\pm1, q\pm1)\Rightarrow (c\pm1, p, q),$$
and for $g\to g\mp2$,
$$(c, g)\Rightarrow(c\pm1, g\mp2)\underset{P_{c\mp}(c\pm1, g\mp2)}\to(c, g\mp2),$$
where the operators for $\Rightarrow$ are to be found.
So, in \S \ref{SE6Shiftcpq} and \ref{SE6Shiftcg}, it is enough to find shift operators
for $(c, p, q)\to(c\pm1, p\mp1, q\mp1)$ and $(c, g)\to(c\pm1, g\mp2)$.

\begin{remark}The obtained shift operators depend on the seven parameters $a,b,\dots,r$ polynomially. The most operators, however, have very long expressions. To make the expressions shorter, it is sometimes useful to use the parameters $e_1,\dots,e_9$ of $E_6$ related to $a,b,\dots,r$ as in Theorem \ref{etoar}. So the results are often written in terms of the $e_j$'s. These expressions may show the hidden symmetry of shift operators, but one must not forget that these expressions are not unique.\end{remark}
For a shift $sh$,
write the shifted operator by putting a letter $s$,
such as $T_s, T_{0s}, B_{1s}$, etc.
Our task is to solve the equation
\begin{equation} \label{tpqt}
 T_s\circ P = Q\circ T.
\end{equation}

\subsubsection{A preparation}

We put
\begin{eqnarray}\label{PQz}
\qquad P= x^2P_{nn} + xP_{n} + P_0 + P_1  \partial+P_2\partial^2, \quad
Q=x^2Q_{nn} + \cdots+Q_2\partial^2.
\end{eqnarray}
The equation $E_6$ is written
as $T:=T_0+T_1\partial+T_2\partial^2+T_3\partial^3$,
the coefficients are given in \eqref{eqT0},\dots,\eqref{eqT3}.
In the following we need to translate the variable $\theta$ and the 
parameters: We denote, say, $B_1[-1]$ for $B_1(\theta\to \theta-1)$ and
$B_2[3]$ for $B_2(\theta\to \theta+3)$. Writing both sides in  $(x,\theta,\partial)$-form, we get

\begin{lemma} \label{syseq}  
The equation \eqref{tpqt} is equivalent to the system of equations
\[E2=0,\  E1=0,\  E0=0,\  Ed1=0,\dots, \ Ed5=0,\]
where
\begin{eqnarray*}
E2 &=& T_{0s}[2] P_{nn}-Q_{nn} T_0,\\
E1 &=& T_{0s}[1] P_n +(\theta+2) T_{1s}[1] P_{nn} -(Q_n T_0+\theta Q_{nn}[-1] T_1[-1]), \\
E0 &=& T_{0s} P_0+(\theta+1) T_{1s} P_n+(\theta+1) (\theta+2) T_{2s} P_{nn} \\
&&\qquad        -(Q_0 T_0+\theta Q_n[-1] T_1[-1]+ \theta (\theta-1) Q_{nn}[-2] T_2[-2])), \\
Ed1 &=& T_{0s} P_1+T_{1s} P_0[1]+(\theta+2) T_{2s} P_n[1]+(\theta+2) (\theta+3) T_{3s} P_{nn}[1] \\
&&\qquad      -(Q_1 T_0[1]+Q_0 T_1+\theta Q_n[-1] T_2[-1]+\theta (\theta-1) Q_{nn}[-2] T_3[-2]),\\
Ed2 &=& T_{0s}P_2+T_{1s} P_1[1]+T_{2s} P_0[2]+(\theta+3) T_{3s} P_n[2]\\
    &&\qquad       -(Q_2T_0[2]+Q_1 T_1[1]+Q_0 T_2+\theta Q_n[-1] T_3[-1]), \\
Ed3 &=&T_{1s}P_2[1]+T_{2s}P_1[2]+ T_{3s} P_0[3]
      -(Q_2T_1[2]+Q_1 T_2[1]+Q_0 T_3), \\
Ed4 &=& T_{2s} P_2[2]+T_{3s} P_1[3]-(Q_2T_2[2]+Q_1 T_3[1]), \\
Ed5&=&T_{3s}P_2[3]-Q_2T_3[2].
\end{eqnarray*}
\end{lemma}

\subsubsection{Shift operators for $a \rightarrow a+1$}\label{SE6Shifta}

Due to the correspondence among parameters as in Theorem \ref{etoar}, the shift $a \rightarrow a+1$
in terms of $e_i$ is expressed as
\[sh_{a+}:= \{e_2\to e_2+1,\ e_3\to e_3+2,\ e_8\to e_8-1,\ e_9\to e_9-2\}.\]
The shifted equation of $E_6$ 
is $T_s =T_{0s}+\cdots+T_{3s}\partial^3,$ where
\begin{eqnarray} 
\quad  T_{0s} &=& (\theta-r)(\theta+1-r)(\theta+2-r)(\theta+e_7)(\theta+e_8-1)(\theta+e_9-2), \label{eqT0s}\\
\quad   T_{1s} &=& (\theta+1-r)(\theta+2-r)B_{1s}, \quad  B_{1s}=B_1(sh_{a+})\label{eqT1s}\\
\quad   T_{2s} &=& (\theta+2-r)B_{2s}, \quad B_{2s}=B_2(sh_{a+})\label{eqT2s}\\
\quad   T_{3s} &=& -(\theta+3-e_1)(\theta+2-e_2)(\theta+1-e_3).\label{eqT3s}
\end{eqnarray}

\begin{enumerate}[leftmargin=*]
\item Rewrite equation $Ed5=0$, by use of \eqref{eqT3} and \eqref{eqT3s}
and get 
\[ (\theta+3-e_1)(\theta+2-e_2)(\theta+1-e_3)P_2[3]=Q_2(\theta+5-e_1)(\theta+5-e_2)(\theta+5-e_3).\]
Since degree$(P_2,\theta)=$degree$(Q_2,\theta)=3$, there is a constant $K$ such that
\[\def\arraystretch{1.2}\setlength\arraycolsep{3pt} \begin{array}{lcl}
   P_2 &=& K(\theta+2-e_1)(\theta+2-e_2)(\theta+2-e_3), \\
   Q_2 &=& K(\theta+3-e_1)(\theta+2-e_2)(\theta+1-e_3).
\end{array}
\]

\item Rewrite equation $Ed4=0$, by use of \eqref{eqT2}, \eqref{eqT3},
\eqref{eqT2s}, and \eqref{eqT3s} and get
\begin{eqnarray*}
&(\theta+3-e_1)(\theta+2-e_2)(\theta+1-e_3)(P_1[3] + K(\theta+4-r)B_2[2]) & \\
&= (\theta+4-e_1)(\theta+4-e_2)(\theta+4-e_3)(Q_1 + K(\theta+2-r)B_{2s}). &
\end{eqnarray*}
Hence, we can set, for some linear form $X$ of $\theta$, as
\[\def\arraystretch{1.2}\setlength\arraycolsep{3pt} \begin{array}{lcl}
 P_1 &=& -K(\theta+1-r)B_2[-1]+(\theta+1-e_1)(\theta+1-e_2)(\theta+1-e_3)X[-3], \\
 Q_1 &=& - K(\theta+2-r)B_{2s} + (\theta+3-e_1)(\theta+2-e_2)(\theta+1-e_3)X.
\end{array}
\]

\item Rewrite $Ed3=0$ and we get
\begin{eqnarray*}
&& (\theta+2-e_2)(\theta+1-e_3)[P_0[3]+K(\theta+3-r)(\theta+4-r)B_1[2]
+(\theta+3-r)B_2[1]X] \\
&& =  (\theta+3-e_2)(\theta+3-e_3)[Q_0 + K(\theta+1-r)(\theta+2-r)B_{1s}
 + (\theta+2-r)B_{2s}X[-1]].
\end{eqnarray*}
Hence, we can set 
\[\def\arraystretch{1.2}\setlength\arraycolsep{3pt} \begin{array}{lcl}
P_0 &=& - K(\theta-r)(\theta+1-r)B_1(\theta-1) - (\theta-r)B_2(\theta-2)X[-3] \\
  &&\quad           + N[-3](\theta-e_2)(\theta-e_3), \\
Q_0 &=& - K(\theta+1-r)(\theta+2-r)B_{1s} - (\theta+2-r)B_{2s}X[-1])  \\
  &&\quad           + N[-3](\theta+2-e_2)(\theta+1-e_3),
\end{array}\]
for some factor $N$.

\item Rewrite $E2=0$ by use of \eqref{eqT0} and \eqref{eqT0s} to get
\begin{eqnarray*}
&  (\theta+2-r)(\theta+3-r)(\theta+4-r)(\theta+e_7+2)(\theta+e_8+1)(\theta+e_9)P_{nn} &\\
& = (\theta-r)(\theta+1-r)(\theta+2-r)(\theta+e_7)(\theta+e_8)(\theta+e_9)Q_{nn}. &
\end{eqnarray*}
Since degree$(P_{nn},\theta)=$degree$(Q_{nn},\theta)=5$, 
there exist constants $M_1$ and $M_0$ such that
\[\def\arraystretch{1.2}\setlength\arraycolsep{3pt} \begin{array}{lcl}
  P_{nn}&=& (M_1\theta+M_0)(\theta-r)(\theta+1-r)(\theta+e_7)(\theta+e_8), \\
  Q_{nn}&=& (M_1\theta+M_0)(\theta+3-r)(\theta+4-r)(\theta+e_7+2)(\theta+e_8+1).
\end{array}\]

\item Rewrite $E1=0$ by use of \eqref{eqT0}, \eqref{eqT1}, \eqref{eqT0s}
and \eqref{eqT1s} and we get
\begin{eqnarray*}
& (\theta+3-r)(\theta+e_7+1)[(\theta+e_9-1)P_{n} -\theta(\theta-r)B_1[-1](M_1(\theta-1)+M_0)]& \\
& =(\theta-r)(\theta+e_7)[(\theta+e_9)Q_{n} -(\theta+2)(\theta+3-r)B_{1s}[1](M_1\theta+M_0)].
\end{eqnarray*}
Hence, we can see that 
$P_n$ is divisible by $\theta-r$ and $Q_n$ is divisible by $\theta+3-r$: we set,
\begin{equation}\label{CPQ}
  P_{n}=(\theta-r){{C}\kern-.2em{P}}4, \qquad    Q_{n}=(\theta+3-r){{C}\kern-.1em{Q}}4,
  \end{equation}
where ${{C}\kern-.2em{P}}4$ and ${{C}\kern-.1em{Q}}4$ are polynomials
  of $\theta$ of degree at most $4$ and, thus, we get a relation
\begin{equation}\label{CP4Q4}
  \begin{array}{c}
    (\theta+e_7+1)((\theta+e_9-1){{C}\kern-.2em{P}}4 -\theta B_1[-1](M_1(\theta -1)+M_0)) \\
    \noalign{\smallskip}
    =(\theta +e_7)((\theta +e_9){{C}\kern-.1em{Q}}4 -(\theta +2)B_{1s}[1](M_1\theta +M_0)).
    \end{array}
  \end{equation}
The factors ${{C}\kern-.2em{P}}4$ and ${{C}\kern-.1em{Q}}4$ will be determined later.  

\item Rewrite $Ed2=0$, by \eqref{CPQ}, we get an identity:
\begin{equation} \label{CP4Q4n}
  \begin{array}{l}
 (\theta +1-e_3)[(\theta +3)(\theta +3-e_1){{C}\kern-.2em{P}}4[2]  \\ \noalign{\smallskip}
      \quad + K(\theta +3-e_1)(\theta +3-r)(\theta +4-r)(\theta +e_7+2)\\
      \quad \hskip36pt \times(\theta +e_8+2)(\theta +e_9+2)  \\
      \noalign{\smallskip}
 \quad + (\theta +3-e_1)(\theta +3-r)XB_1[1]+ N[-1]B_2]  \\ \noalign{\smallskip}
= (\theta +2-e_3)[\theta (\theta +2-e_1){{C}\kern-.1em{Q}}4[-1]  \\ \noalign{\smallskip}
  \quad +K(\theta +2-e_1)(\theta -r)(\theta +1-r)(\theta +e_7)(\theta +e_8-1)(\theta +e_9-2)  \\
  \noalign{\smallskip}
  \quad +(\theta +2-e_1)(\theta +1-r)B_{1s}X[-2]+ NB_{2s}].
\end{array}
\end{equation}

\end{enumerate}

Writing as $C\!P4 = p_4\theta ^4+\cdots+p_0$, $C\!Q4 = q_4\theta ^4+\cdots+q_0$,
we now have the forms of all $P$'s and $Q$'s 
and we substitute them into the equations
$Ed1=0$, $E0=0$, \eqref{CP4Q4}, and \eqref{CP4Q4n}
to get the values of
$\{K, M_1, M_0, X, N, p_4, \dots, p_0, q_4, \dots, q_0\}$.
Here, we solve them by assuming $N$ is a constant.
The result is as follows: 
\begin{eqnarray*} 
P_{nn}&=&(M_1\theta +M_0)(\theta -r)(\theta +1-r)(\theta +e_7)(\theta +e_8), \\
Q_{nn}&=&(M_1\theta +M_0)(\theta +3-r)(\theta +4-r)(\theta +e_7+2)(\theta +e_8+1), \\
P_0&=&-K(\theta -r)(\theta +1-r)B_1[-1]-(\theta -r)X[-3]B_2[-2]+ N(\theta -e_2)(\theta -e_3), \\
Q_0&=&-K(\theta +1-r)(\theta +2-r)B_{1s}-(\theta +2-r)X[-1]B_{2s} \\
&&\qquad +N(\theta +2-e_2)(\theta +1-e_3), \\
P_1&=&-K(\theta +1-r)B_2[-1]+ (\theta +1-e_1)(\theta +1-e_2)(\theta +1-e_3)X[-3], \\
Q_1&=&- K(\theta +2-r)B_{2s} + (\theta +3-e_1)(\theta +2-e_2)(\theta +1-e_3)X, \\      
P_2&=&K(\theta +2-e_1)(\theta +2-e_2)(\theta +2-e_3), \\
Q_2&=&K(\theta +3-e_1)(\theta +2-e_2)(\theta +1-e_3), 
\end{eqnarray*}
where
\begin{eqnarray*}
K &=& -e_1e_8+a(e_9+p+q+r)-(4a+3b+4c+2g+q+6), \\
M_1 &=& (e_8-1)(e_9-r-1), \\
M_0 &=& (e_8-1)((e_9-1)(e_9-r)+r(r-2)), \\
N &=& r(r-1)(e_8-1)(e_9-1)(e_9-2), \\
X &=&  -M_1\theta +(e_8-1)(-(r+2)e_9+4r+2).
\end{eqnarray*}
Instead of writing $P_n$ and $Q_n$ explicitly by use of
  $p_4, \dots, q_0$, 
by referring to \eqref{CPQ}, we give expressions
\begin{eqnarray*}
&&(\theta +1)(\theta +1-e_1){{C}\kern-.2em{P}}4 = N((\theta -e_3)U_2[-3]-B_2[-2]) \\
&&\qquad\qquad   - K(\theta +1-e_1)(\theta +1-r)(\theta +2-r)(\theta +e_7)(\theta +e_8)(\theta +e_9) \\
&&\qquad\qquad    - (\theta +1-e_1)(\theta +1-r)X[-2]B_1[-1], \\
\noalign{\smallskip}
&&(\theta +1)(\theta +3-e_1){{C}\kern-.1em{Q}}4 = N((\theta +2-e_3)U_2-B_{2s}[1]) \\
&&\qquad\qquad    -K(\theta +3-e_1)(\theta +1-r)(\theta +2-r)(\theta +e_7+1)(\theta +e_8)(\theta +e_9-1) \\
&&\qquad\qquad         -(\theta +3-e_1)(\theta +2-r)X[-1]B_{1s}[1],
\end{eqnarray*}
where
\[U_2:=(\theta +2-e_2)(\theta +2+e_9) + K + e_1e_8.\]
Thus, we can compute ${C}\kern-.2em{P}4$ and ${C}\kern-.2em{Q}4$.
\begin{remark}To make the expression shorter, we also used notation for the exponents $$e_1=p+r+1,\dots,e_9=-2(a+b+c)-p-q-r-g-3.$$
\end{remark}

\begin{remark} We need much space to give expressions of $P_n$ and $Q_n$,
    which are therefore not in this paper; They are listed
    in the file named {\it SE6PQ.txt} in the site mentioned in Introduction.
\end{remark} 

\subsubsection{Shift operators for $a \rightarrow a-1$} \label{SE6Shiftan}

We can compute the shift operators $P_{a-}$ and $Q_{a-}$
for $a\rightarrow a-1$, similarly. The result is the following.
\begin{eqnarray*}
  P_{nn}&=&M(\theta -r)(\theta +1-r)(\theta +e_7)(\theta +e_8)(\theta +e_9), \\
  Q_{nn}&=&M(\theta +3-r)(\theta +4-r)(\theta +e_7+2)(\theta +e_8+3)(\theta +e_9+4), \\
  P_n&=&M\theta (\theta -r)B_1[-1]+(\theta -r)(\theta +e_7)(\theta +e_8)(\theta +e_9)X, \\
  Q_n&=&M(\theta +2)(\theta +3-r)B_{1s}[1]\\
  && \qquad +(\theta +3-r)(\theta +e_7+1)(\theta +e_8+2)(\theta +e_9+3)X, \\
  P_0&=&\theta X[-1]B_1[-1]+M\theta (\theta -1)B_2[-2]+N(\theta +e_8)(\theta +e_9), \\
  Q_0&=&(\theta +1)X B_{1s}+M(\theta +1)(\theta +2)B_{2s}+N(\theta +e_8+1)(\theta +e_9+2), \\   
  P_2&=&(K_1(\theta -3)+K_0)(\theta +2-e_1)(\theta +2-e_2), \\
  Q_2&=&(K_1\theta +K_0)(\theta +3-e_1)(\theta +4-e_2), 
\end{eqnarray*}
where
\begin{eqnarray*}
K_1 &=& -(e_2-r-1)(e_3-2r-1), \\
K_0 &=& (e_2-r-1)((e_3-r-3)^2+3r-4), \\
M &=& 2a^2+2ab+3ag-pa-qa+gb-2c^2+cg-3cp-cq\\
&&\qquad +g^2-p^2-qp-4c+g-3p-q-2, \\
N &=& r(r-1)(e_2-r-1)(e_3-r-1)(e_3-r-2), \\
X &=& -K_1\theta  -(r-1)(e_2-r-1)(2e_3-3r-2).
\end{eqnarray*}
$P_1$ and $Q_1$ are given as
\begin{eqnarray*}
&& (\theta -r)(\theta +e_7)P_1=  N(B_1+(\theta +e_9+1)V_2) + \theta X[-1]B_2[-1](\theta +e_7)  \\
&&\qquad  + \theta (\theta -1)(\theta +1-e_1)(\theta +1-e_2)(\theta +1-e_3)M(\theta +e_7) \\
  \noalign{\smallskip}
  && (\theta +3-r)(\theta +e_7+1)Q_1 =  N(B_{1s}+(\theta +e_9+2)V_2)\\
  && \qquad + (\theta +2)X[1]B_{2s}(\theta +e_7+1) \\
&&\qquad + (\theta +2)(\theta +3)(\theta +3-e_1)(\theta +4-e_2)(\theta +5-e_3)M(\theta +e_7+1),
\end{eqnarray*}
where
\[ V_2= (\theta +e_8)(\theta -e_3+4)+(e_2-r-2)(e_7+r-2)-M+b-c. \]

\subsubsection{Shift operators for $(c, p, q)\rightarrow  (c\pm 1, p\mp 1, q\mp 1)$}\label{SE6Shiftcpq}

The shift 
$(c, p, q)\rightarrow (c+1, p-1, q-1)$ in terms of parameters $e_i$ is given as
\[shcpq:=\{e_1\to e_1-1,\, e_3\to e_3+1,\, e_4\to e_4-1,\, e_6\to e_6+1\}.\]
The shifted  equation is denoted by $T_s =T_{0s}+\cdots+T_{3s}\partial^3,$ as in the previous subsubsection.
The procedure to get the components $P_{nn}$, $Q_{nn}$, $\dots$ is
similar to that in \S \ref{SE6Shifta}. A system of equations is
the same as in Lemma \ref{syseq}. We give only a sketch of computation.

\begin{enumerate}[leftmargin=*]

\item Equation $Ed5=0$ is written as
\[(\theta +4-e_1)(\theta +3-e_2)(\theta +2-e_3)P_2[3]=(\theta +5-e_1)(\theta +5-e_2)(\theta +5-e_3)Q_2,\]
and this shows that
\[\def\arraystretch{1.2}\setlength\arraycolsep{3pt} \begin{array}{lcl}
   P_2&=& K(\theta +2-e_1)(\theta +2-e_2)(\theta +2-e_3), \\
   Q_2&=& K(\theta +4-e_1)(\theta +3-e_2)(\theta +2-e_3),
\end{array}\]
for a constant $K$.

\item Equation $E2=0$ implies
\begin{eqnarray*}
&&(\theta +2-r)(\theta +3-r)(\theta +4-r)(\theta +e_7+2)(\theta +e_8+2)(\theta +e_9+2)P_{nn} \\
&&\quad  = (\theta -r)(\theta +1-r)(\theta +2-r)(\theta +e_7)(\theta +e_8)(\theta +e_9)Q_{nn},
\end{eqnarray*}
which implies existence of a constant $M$ such that
\[\def\arraystretch{1.2}\setlength\arraycolsep{3pt} \begin{array}{lcl}
  P_{nn} &=& M(\theta -r)(\theta +1-r)(\theta +e_7)(\theta +e_8)(\theta +e_9), \\
  Q_{nn} &=& M(\theta +3-r)(\theta +4-r)(\theta +e_7+2)(\theta +e_8+2)(\theta +e_9+2).
\end{array}
\]

\item Equation $E1=0$ implies
\begin{eqnarray*}
&& (\theta +1-r)(\theta +2-r)(\theta +3-r)(\theta +e_7+1)(\theta +e_8+1)(\theta +e_9+1)P_n \\
&&\quad  +(\theta +2)(\theta +2-r)(\theta +3-r)B_{1s}[1]\\
&&\qquad  \times M(\theta -r)(\theta +1-r)(\theta +e_7)(\theta +e_8)(\theta +e_9) \\
&& =(\theta -r)(\theta +1-r)(\theta +2-r)(\theta +e_7)(\theta +e_8)(\theta +e_9)Q_n \\
&&\quad   +\theta (\theta -r)(\theta +1-r)B_1[-1]\\
&& \qquad  \times M(\theta +2-r)(\theta +3-r)(\theta +e_7+1)(\theta +e_8+1)(\theta +e_9+1),
\end{eqnarray*}
which leads to
\begin{eqnarray*}
&& (\theta +e_7+1)(\theta +e_8+1)(\theta +e_9+1)(\theta +3-r)
    (P_n  - M\theta (\theta -r)B_1[-1]) \\
&&\quad  =(\theta +e_7)(\theta +e_8)(\theta +e_9)(\theta -r)
    (Q_n  - M(\theta +2)(\theta +3-r)B_{1s}[1])
\end{eqnarray*}
Hence, for a linear factor $X$, we have 
\begin{eqnarray*}
 P_n &=& M \theta (\theta -r)B_1[-1] + (\theta +e_7)(\theta +e_8)(\theta +e_9)(\theta -r)X, \\
 Q_n &=& M (\theta +2)(\theta +3-r)B_{1s}[1] + (\theta +e_7+1)(\theta +e_8+1)(\theta +e_9+1)\\
 &&\qquad\qquad \times (\theta +3-r)X.
\end{eqnarray*}

\item Equation $Ed4=0$ is 
\begin{eqnarray*}
&& K (\theta +2-r)B_{2s}(\theta +4-e_1)(\theta +4-e_2)(\theta +4-e_3) \\
&&\quad    -(\theta +4-e_1)(\theta +3-e_2)(\theta +2-e_3)P_1[3] \\
&& = K (\theta +4-r)B_2[2](\theta +4-e_1)(\theta +3-e_2)(\theta +2-e_3) \\
&&\quad   -(\theta +4-e_1)(\theta +4-e_2)(\theta +4-e_3)Q_1.
\end{eqnarray*}
From this, we conclude the existence of the factor $V$
of degree at  most 3 such that
\begin{eqnarray*}
  P_1 &=& (\theta +1-e_2)(\theta +1-e_3)V[-3]-K (\theta +1-r)B_2[-1], \\
  Q_1 &=& (\theta +3-e_2)(\theta +2-e_3)V -K (\theta +2-r)B_{2s}.
\end{eqnarray*}

\item Equation $Ed3=0$ is 
\begin{eqnarray*}
&& K(\theta +1-r)(\theta +2-r)B_{1s}(\theta +3-e_1)(\theta +3-e_2)(\theta +3-e_3) \\
&&\quad  +(\theta +2-r)B_{2s}((\theta +3-e_2)(\theta +3-e_3)V[-1]-K(\theta +3-r)B_2[1]) \\
&&\quad  -(\theta +4-e_1)(\theta +3-e_2)(\theta +2-e_3)P_0[3] \\
&&= K(\theta +3-r)(\theta +4-r)B_1[2](\theta +4-e_1)(\theta +3-e_2)(\theta +2-e_3) \\
&&\quad  +(\theta +3-r)B_2[1]((\theta +3-e_2)(\theta +2-e_3)V -K(\theta +2-r)B_{2s}) \\
&&\quad  -(\theta +3-e_1)(\theta +3-e_2)(\theta +3-e_3)Q_0.
\end{eqnarray*}
Computation shows that 
\begin{eqnarray*}
&&  (\theta +3-e_3)[(\theta +3-e_1)Q_0\\
&&\qquad   + K(\theta +1-r)(\theta +2-r)B_{1s}(\theta +3-e_1)+ (\theta +2-r)B_{2s}V[-1]] \\
  && = (\theta +2-e_3)[(\theta +4-e_1)P_0[3]\\
&&\qquad    + K(\theta +3-r)(\theta +4-r)B_1[2](\theta +4-e_1)+ (\theta +3-r)B_2[1]V],
\end{eqnarray*}
from which we conclude the existence of a factor $W$ such that
\begin{eqnarray}
&& (\theta +4-e_1)P_0[3]+ K(\theta +3-r)(\theta +4-r)(\theta +4-e_1)B_1[2] \label{w2} \\
  &&\qquad  + (\theta +3-r)B_2[1]V=(\theta +3-e_3)W, \nonumber \\
  \noalign{\smallskip}
  && (\theta +3-e_1)Q_0+ K(\theta +1-r)(\theta +2-r)(\theta +3-e_1)B_{1s} \label{w1} \\
&& \qquad + (\theta +2-r)B_{2s}V[-1]=(\theta +2-e_3)W. \nonumber
\end{eqnarray}
These identities will be used to get expressions of $P_0$
and $Q_0$. 

\item Now, we have fixed the expressions of $P_{nn}$, $Q_{nn}$, $\dots$,
and we solve the remaining equations 
$Ed3=0$, $Ed2=0$, $Ed1=0$ and $E0=0$, by computer assistance.  
The result is the following:
\begin{eqnarray*}
 P_2&=& K(\theta +2-e_1)(\theta +2-e_2)(\theta +2-e_3), \\
 Q_2 &=&K(\theta +4-e_1)(\theta +3-e_2)(\theta +2-e_3), \\
 P_{nn} &=& M(\theta -r)(\theta +1-r)(\theta +e_7)(\theta +e_8)(\theta +e_9), \\
 Q_{nn} &=& M(\theta +3-r)(\theta +4-r)(\theta +e_7+2)(\theta +e_8+2)(\theta +e_9+2), \\
 P_n &=& M \theta (\theta -r)B_1[-1] -K(\theta -r)(\theta +2-2r)(\theta +e_7)(\theta +e_8)(\theta +e_9), \\
 Q_n &=& M (\theta +2)(\theta +3-r)B_{1s}[1] - K(\theta +3-r)(\theta +2-2r)(\theta +e_7+1)\\
 &&\times(\theta +e_8+1)(\theta +e_9+1), \\
  P_1 &=& -K (\theta +1-r)B_2[-1]+(\theta +1-e_2)(\theta +1-e_3)V[-3], \\
  Q_1 &=& -K (\theta +2-r)B_{2s}+(\theta +3-e_2)(\theta +2-e_3)V,
\end{eqnarray*}
where
\begin{eqnarray*}
  K &=& e_1-r-1, \\
  M &=& e_1+e_4-2r-2, \\
  V &=& -M(\theta +2)(\theta +3-K)+Kr(M+r-1).
\end{eqnarray*}
$P_0$ and $Q_0$ are given by solving (\ref{w2}) and (\ref{w1}), where
$$W = r(r-1)(e_1-r-1)W_1,$$ 

\[\begin{array}{l}
 W_1= (\theta -e_2+3)(\theta -M-r+3)(\theta -M-r+2) \\
 \hskip40pt      - (e_4-r-1)(M+e_7+r-1)(M+e_9+r-1)/2.
\end{array}\]

\end{enumerate}

We here give a remark on the shift $(c,p,q)\rightarrow (c-1,p+1,q+1)$,
which is equal to the shift
\[\{e_1\to e_1+1, e_3\to e_3-1, e_4\to e_4+1, e_6\to e_6-1\}.\]
Since this is equal to the shift $shcpq$ above up to
the exchange $\{e_1\leftrightarrow e_3, e_4\leftrightarrow e_6\}$, 
and the equation $E_6$ is invariant relative to this exchange,
the corresponding shift operators for $S\!E_6$ are obtained
from the above operators by this exchange.

\subsubsection{Shift operators for $(c, g)\rightarrow (c\pm 1, g\mp 2)$}\label{SE6Shiftcg}

We consider the shift $(c, g)\rightarrow (c+1, g-2)$, which is equal to
\[shcg:=\{e_2\to e_2+1,\, e_5\to e_5+1,\,  e_7\to e_7-2\}.\]
The shift operators are denoted by $P_{cg}$ and $Q_{cg}$.
We follow the procedure of getting the components $P_{nn}$, $Q_{nn}$, $\dots$ 
as before.

\begin{enumerate}[leftmargin=*]
\item Equation $Ed5=0$ implies 
\[\def\arraystretch{1.2}\setlength\arraycolsep{3pt} \begin{array}{lcl}
   P_2 &=& K(\theta +2-e_1)(\theta +2-e_2)(\theta +2-e_3), \\
   Q_2 &=& K(\theta +3-e_1)(\theta +2-e_2)(\theta +3-e_3),
\end{array}\]
for some constant $K$.

\item Equation $E2=0$ implies
\begin{eqnarray*}
&  (\theta +2-r)(\theta +3-r)(\theta +4-r)(\theta +e_7)(\theta +e_8+2)(\theta +e_9+2)P_{nn}& \\
&  = (\theta -r)(\theta +1-r)(\theta +2-r)(\theta +e_7)(\theta +e_8)(\theta +e_9)Q_{nn},&
\end{eqnarray*}
which reduces to
\begin{eqnarray*}
&&  (\theta +3-r)(\theta +4-r)(\theta +e_8+2)(\theta +e_9+2)P_{nn} \\
&&\quad  = (\theta -r)(\theta +1-r)(\theta +e_8)(\theta +e_9)Q_{nn}.
\end{eqnarray*}
Hence, for some constants $M_1$ and $M_0$,
\[\def\arraystretch{1.2}\setlength\arraycolsep{3pt} \begin{array}{lcl}
  P_{nn} &=& (M_1\theta +M_0)(\theta -r)(\theta +1-r)(\theta +e_8)(\theta +e_9), \\
  Q_{nn} &=& (M_1\theta +M_0)(\theta +3-r)(\theta +4-r)(\theta +e_8+2)(\theta +e_9+2).
\end{array}\]

\item Equation $E1=0$ is rewritten as 
\begin{eqnarray*}
&& (\theta +3-r)(\theta +e_8+1)(\theta +e_9+1) \\
&&\quad \times ((\theta +e_7-1)P_n - \theta (\theta -r)B_1[-1](M_1(\theta -1)+M_0))\\
&& =(\theta -r)(\theta +e_8)(\theta +e_9) \\
&&\quad \times ((\theta +e_7)Q_n -(\theta +2)(\theta +3-r)B_{1s}[1](M_1\theta +M_0)),
\end{eqnarray*}
which implies that $P_n$ is divisible by $(\theta -r)$ and $Q_n$
by $(\theta +3-r)$, and we set 
\[ P_{n}=(\theta -r){{C}\kern-.2em{P}}4, \qquad    Q_{n}=(\theta +3-r){{C}\kern-.1em{Q}}4,\]
as in \eqref{CPQ}; then for a polynomial $U_3$ of degree at most 3, we have
\begin{equation} \label{u3pq}
  \begin{array}{c}
    (\theta +e_7-1){{C}\kern-.2em{P}}4 = \theta (M_1(\theta -1)+M_0)B_1[-1]+(\theta +e_8)(\theta +e_9)U_3, \\
    \noalign{\smallskip}
    (\theta +e_7){{C}\kern-.1em{Q}}4 = (\theta +2)(M_1\theta +M_0)B_{1s}[1] +(\theta +e_8+1)(\theta +e_9+1)U_3.
  \end{array}
  \end{equation}

\item Equation $Ed4=0$  is rewritten as
\begin{eqnarray*}
&& (\theta +4-e_1)(\theta +4-e_2)(\theta +4-e_3)(Q_1+K(\theta +2-r)B_{2s}) \\
&& = (\theta +3-e_1)(\theta +2-e_2)(\theta +3-e_3)(P_1[3]+K(\theta +4-r)B_2[2]).
\end{eqnarray*}
Hence, there is a linear factor $X$ such that
\[\def\arraystretch{1.2}\setlength\arraycolsep{3pt} \begin{array}{lcl}
 P_1 &=& -K(\theta +1-r)B_2[-1] + (\theta +1-e_1)(\theta +1-e_2)(\theta +1-e_3)X[-3], \\
 Q_1 &=&  -K(\theta +2-r)B_{2s} + (\theta +3-e_1)(\theta +2-e_2)(\theta +3-e_3)X.
\end{array}\]

\item Equation $Ed3=0$ is rewritten as 
\begin{eqnarray*}
&&  (\theta +2-e_2)(P_0[3]+K(\theta +3-r)(\theta +4-r)B_1[2] +(\theta +3-r)X B_2[1])\\
&&= (\theta +3-e_2)(Q_0+K(\theta +1-r)(\theta +2-r)B_{1s}+(\theta +2-r)X[-1]B_{2s}).
\end{eqnarray*}
Hence, by consideration of degree, we see that there exists
a polynomial $V$ in $\theta $ of degree at most 4
so that $P_0$ and $Q_0$ are expressed as
\[\def\arraystretch{1.2}\setlength\arraycolsep{3pt} \begin{array}{lcl}
  P_0&=& -K(\theta -r)(\theta +1-r)B_1[-1]-(\theta -r)X[-3]B_2[-2]+(\theta -e_2)V[-3],\\
  Q_0&=&  -K(\theta +1-r)(\theta +2-r)B_{1s}-(\theta +2-r)X[-1]B_{2s}+(\theta +2-e_2)V. 
\end{array}\]

\item Now, with the forms of $P_{nn}$, $\dots$ given above and by use of
  the expression as $C\!P4 = p_4\theta ^4+\cdots+p_0$, $C\!Q4 = q_4\theta ^4+\cdots+q_0$,
  we solve the
equations \eqref{u3pq}, $E0=0$, $Ed1=0$, and $Ed2=0$, and we have the result
as follows:
\begin{eqnarray*}
  P_2&=&K(\theta +2-e_1)(\theta +2-e_2)(\theta +2-e_3), \\
  Q_2 &=& K(\theta +3-e_1)(\theta +2-e_2)(\theta +3-e_3), \\
  P_{nn}&=&(M_1\theta +M_0)(\theta -r)(\theta +1-r)(\theta +e_8)(\theta +e_9), \\
  Q_{nn}&=&(M_1\theta +M_0)(\theta +3-r)(\theta +4-r)(\theta +e_8+2)(\theta +e_9+2), \\
  P_0&=& -K(\theta -r)(\theta +1-r)B_1[-1]-(\theta -r)X[-3]B_2[-2]+(\theta -e_2)V[-3], \\
  Q_0&=& -K(\theta +1-r)(\theta +2-r)B_{1s}-(\theta +2-r)X[-1]B_{2s}+(\theta +2-e_2)V, \\
  P_1&=& -K(\theta +1-r)B_2[-1] + (\theta +1-e_1)(\theta +1-e_2)(\theta +1-e_3)X[-3], \\
  Q_1&=& -K(\theta +2-r)B_{2s} + (\theta +3-e_1)(\theta +2-e_2)(\theta +3-e_3)X, 
\end{eqnarray*}
where
\begin{eqnarray*}
K &=& (e_1-1)(e_7-e_8-1)-(e_1+e_6+e_7-r-3)(e_1+e_4+e_7-r-3)/2 ,\\
M_1 &=& (e_7-e_8-1)(r+1-e_7), \\
M_0 &=& -(e_7-e_8-1)(e_7(e_7-r-1)+r^2-r), \\
X&=&  -M_1\theta  -(e_7-e_8-1)(-(r+2)e_7+4r+2), \\
V &=& -r(r-1)(e_7-1)(e_7-2)(K+(e_7-e_8-1)(\theta -e_1-e_3+2)).
\end{eqnarray*}
The factor $U_3$ defined in \eqref{u3pq}
is computed by use of ${{C}\kern-.2em{P}}4$ which we did not write above, as
\begin{eqnarray*}
  U_3 &=& -K(\theta +e_7-1)(\theta -r+1)(\theta +e_7-r+1) \\
  &&\quad +K(r-1)(e_7-1)(\theta -r-1+e_7)\\
&&\quad    -r(r-1)(e_7-1)(e_7-2)(e_7-e_8-1).
\end{eqnarray*}
This gives ${{C}\kern-.2em{P}}4$ and ${{C}\kern-.1em{Q}}4$ conversely, and also $P_n$ and $Q_n$.
\end{enumerate}

Let us next consider the shift 
$(c, g)\rightarrow (c-1, g+2)$, which is equal to the shift
\[shrcg:=\{e_2\to e_2-1,\, e_5\to e_5-1,\, e_7\to e_7+2\}.\]
We denote the shift operators by $P\! R_{cg}$ and $Q\!R_{cg}$.
The procedure is very similar to the previous case
and we obtain the result as follows.
\begin{eqnarray*}
  P_2&=&K(\theta +2-e_1)(\theta +2-e_2)(\theta +2-e_3),\\
  Q_2&=&K(\theta +3-e_1)(\theta +4-e_2)(\theta +3-e_3), \\
  P_{nn}&=&M(\theta -r)(\theta +1-r)(\theta +e_7)(\theta +e_8)(\theta +e_9), \\
  Q_{nn}&=&M(\theta +3-r)(\theta +4-r)(\theta +e_7+4)(\theta +e_8+2)(\theta +e_9+2), \\
  P_n&=& M \theta (\theta -r)B_1[-1] -K(\theta -r)(\theta +2-2r)(\theta +e_7)(\theta +e_8)(\theta +e_9), \\
  Q_n&=& M(\theta +2)(\theta +3-r)B_{1s}[1] -K(\theta +2-2r)(\theta +3-r)(\theta +e_7+3)\\
  &&\qquad \times(\theta +e_8+1)(\theta +e_9+1), \\
  P_0&=&-K\theta (\theta +1-2r)B_1[-1] + M\theta (\theta -1)B_2[-2] + Kr(r-1)(\theta +e_7)N, \\
  Q_0&=&-K(\theta +1)(\theta +2-2r)B_{1s}+ M(\theta +1)(\theta +2)B_{2s} \\
  &&\qquad +Kr(r-1)(\theta +e_7+2)N, \\
  P_1&=& -K(\theta +1-r)B_2[-1] + (\theta +1-e_1)(\theta +1-e_3)U_2[-3], \\
  Q_1&=& -K(\theta +2-r)B_{2s} + (\theta +3-e_1)(\theta +3-e_3)U_2, 
\end{eqnarray*}
where 
\begin{eqnarray*}
K &=& e_2-r-1, \\
M  &=& e_2+e_5-2r-2, \\
N &=& (\theta +e_8)(\theta -e_2-g+1)-(e_5-r-1)(g+1),\\
U_2&=& -M\theta ^2+ M(K-5)\theta  + (r+2)K(M+r-3)-6(M-K), 
\end{eqnarray*}

\subsubsection{S-values and reducibility conditions}\label{SE6ShiftS}

By use of the shift operators obtained above, we can compute
  $S$-values as follows.
Recall that, for example, the S-value $Sv_{a-}$ was defined by the composition: 
\[P_{a+}(a=a-1)\circ P_{a-}(a):{\rm Sol}(a)\rightarrow {\rm Sol}(a-1)\rightarrow {\rm Sol}(a).\]

\begin{prp}\label{S-valuesSE6}
The S-values of the shift operators for $a,b,c,g,p,q,r$:
\[\def\arraystretch{1.2}\setlength\arraycolsep{3pt} \begin{array}{rcl}
Sv_{a-} &=& -r^2(r-1)^2a(2a+g)(e_2-r-1)(a+b+c+g+1)\\
&&      \times(2a+2b+2c+g+2)(e_3-r-1)(e_3-r-2)e_8e_9(e_9+1), \\
\noalign{\smallskip}
Sv_{b-} &=& -r^2(r-1)^2b(2b+g)(e_5-r-1)(a+b+c+g+1)\\
&&     \times(2a+2b+2c+g+2)(e_6-r-1)(e_6-r-2)e_8e_9(e_9+1), \\
\noalign{\smallskip}
Sv_{c-} &=& r^2(r-1)^2c(2c+g)(e_2-r-1)(e_5-r-1)(a+b+c+g+1)\\
&&\times(2a+2b+2c+g+2)(e_3-r-1)(e_3-r-2)(e_6-r-1)\\
&&\times (e_6-r-2)e_7(e_7+1)e_8(e_8+1)e_9(e_9+1), \\
\noalign{\smallskip}
Sv_{g-} &=& r^2(r-1)^2(g-1)(2a+g)(2b+g)(2c+g)(a+b+c+g)\\
&& \times(a+b+c+g+1)(2a+2b+2c+g+2)(e_3-1-r)(e_3-2-r)\\
&&\times (e_6-1-r)(e_6-2-r)e_8(e_8+1)e_9(e_9+1), \\
\noalign{\smallskip}
Sv_{p-} &=& -r^2(r-1)^2p(e_2-r-1)(e_3-r-1)e_7e_8e_9, \\
\noalign{\smallskip}
Sv_{q-} &=& -r^2(r-1)^2q(e_5-r-1)(e_6-r-1)e_7e_8e_9, \\
\noalign{\smallskip}
Sv_{r-} &=& r(r-1)(r-2)e_7e_8e_9.
\end{array}\]
\end{prp}
Recall that the equation is reducible when the S-value of a shift operator vanishes (Proposition \ref{Sred}), and that if an equation is connected by a shift operator with a reducible equation, it is also reducible (Theorem \ref{red_atoap}). Each of the parameters $a,\dots, r$ of $S\!E_6$ admits a shift operator (Theorem \ref{shiftopSE6}). 
So we have the following theorem.

\begin{thm}\label{SE6redcond} If one of 
\begin{gather*}
  -a,\ -b,\  -c,\quad a+g/2+1,\ b+g/2+1,\ c+g/2+1,\\
  -(\delta+g/2+1),\ \delta+g+2,\quad 1/2-g/2,\\ 
e_i-r\ (1\le i\le6),\ e_7,\  e_8,\  e_9,\quad r
\end{gather*}
is an integer, then the equation $S\!E_6$ is reducible.
\end{thm}
A reason why not $a$, $a+g/2,\dots$ but $-a$, $a+g/2+1,\dots$ 
is to make
  statements in Theorem \ref{SE6A17} and Proposition \ref{redSE6}
  simpler.

  
\section{Reducible $S\!E_6$}\label{RedSE6}

In Theorem \ref{SE6redcond}, the conditions on the third line come from those of $E_6$ (Theorem \ref{redcondE6}), the reducible types are $\{1,5\}$ and $\{1,1,1,3\}$ (Propositions \ref{E6red_e9}, \ref{prp1113}).
In this section we see that, for the conditions on the first and second lines, the reducible type is $\{2,4\}$. 

\subsection{Dividing $S\!E_6$ by the hypergeometric operator $E_2$ from the right}\label{EarGauss}

\begin{thm}\label{SE6A17}Let $\Delta$ be one of
\begin{gather*}
  -a,\ -b,\  -c,\quad a+g/2+1,\ b+g/2+1,\ c+g/2+1,\\
  -(\delta+g/2+1),\ \delta+g+2,\quad 1/2-g/2.
\end{gather*}
$S\!E_6(a,b,c,g,p,q,r)$ is divisible by the hypergeometric operator $E_2$ from the right if $\Delta=0$. 
The quotient is ${\rm ST}_4$ $($see \S 1.1$)$.
\end{thm}
Proof of Theorem \ref{SE6A17} goes as follows: We
assume  $$S\!E_6=T_0(\theta)+T_1(\theta)\partial+T_2(\theta)\partial^2+T_3(\theta)\partial^3$$
is divisible from the right by the hypergeometric operator with parameter $(A,B;C)$ 
$$E_2=E(A,B;C):=E_0(\theta )+E_1(\theta )\partial,\quad E_0=(\theta +A)(\theta +B),\ E_1=-(\theta +C),$$
and write the quotient as
$$Q:=Q_0(\theta)+Q_1(\theta)\partial+Q_2(\theta)\partial^2.$$
We solve the equation $Q\circ E_2=S\!E_6$:

\[\def\arraystretch{1.2}
\begin{array}{lll}
       &\ \ \ Q_0E_0(\theta)&=T_0,\\
  Q_0E_1(\theta)&+\ Q_1E_0(\theta +1)&=T_1,\\
  Q_1E_1(\theta +1)&+\ Q_2E_0(\theta +2)&=T_2,\\
  Q_2E_1(\theta +2)&&=T_3.\end{array}
\]
Rewrite the exponents of $S\!E_6$ as
\[\def\arraystretch{1.2}
\begin{array}{l}
(y_{00},\, \dots,\, y_{05})=(0,\, 1,\, 2,\, e_1,\, e_2,\, e_3),\\
(y_{10},\, \dots,\, y_{15})=(0,\, 1,\, 2,\, e_4,\, e_5,\, e_6),\\
(y_{20},\, \dots,\, y_{25})=(-r,\, 1-r,\, 2-r,\, e_7,\, e_8,\, e_9),
\end{array}
\]
where $e_1$, ..., $e_9$ are given in Theorem \ref{etoar}. Then 
\begin{itemize}
\item Since we have
$$T_0= (\theta +y_{20})\cdots(\theta +y_{25}),\quad E_0|T_0,$$
$A$ (and $B$) must be one of
$\{y_{20},\, y_{21},\, y_{22},\, y_{23},\, y_{24},\, y_{25}\}$.
\item Since we have
$$T_3=-(\theta +3-y_{03})(\theta +3-y_{04})(\theta +3-y_{05}),\quad E_1(\theta +2)|T_3,$$
$C$ must be one of $\{1-y_{03},\,  1-y_{04},\,  1-y_{05}\}$,
\end{itemize}
Now for each case we check the remaining two equations
and get the nine possible pairs $(Q,E_2)$ under the condition $\Delta=0$. For example, when $\Delta=-a$, $Q$ and $E_2$ have the following Riemann schemes.
\par\smallskip\noindent
$$ \left(\begin{array}{ccccccccccc}
x=0:&  0&& 1&& y_{04}-1&& y_{05}-1\\              
x=1:&  0&& 1&& y_{13}-1&& y_{14}-1\\             
x=\infty: & \ y_{21}&&\  y_{22}&& y_{23}&& y_{24} \end{array} \right),  \qquad           
\left(\begin{array}{cccccc}
x=0:&  0&  y_{03}    \\    
x=1:&  0&  y_{15}\\
x=\infty:& \  y_{20}\ & y_{25}   \end{array} \right).   $$          

\par\noindent  
The Riemann scheme of $Q$ (note $(y_{21},y_{22})=(1-r,2-r)$)
tell that $Q$ have no accessory parameters and so they are ${\rm ST}_4$.

\begin{remark}When $S\!E_6$ is divisible by $E_2$ from the right, the adjoint  $S\!E^*_6$, which is again an $S\!E_6$, is divisible from the left by the adjoint $E^*_2$, which is again an $E_2$; the quotient is an ${\rm ST}_4$.
\end{remark}
  
\subsection{Reducible cases of $S\!E_6$}\label{RedSE6fact}
When $S\!E_6$ is reducible as given in Theorems \ref{SE6redcond} and \ref{SE6A17}, we show the type of factorization. 
\begin{prp}\label{redSE6}{$(1)$} Recall $\Delta$ in Theorem \ref{SE6A17}.
  If $\Delta \in {\bf Z}$, then $S\!E_6$ factorizes as follows:
  $$\begin{array}{ccccccc}
    \Delta=\ &\ \dots\ &\ -1\ &\ 0\ &\ 1\ &\ 2\ &\ \dots\\
    \ &\ \dots\ &\ [42]\ &\ [42]A0\ &\ [24]A0\ &\ [24]\ &\ \dots\end{array}$$
    When $\Delta=0,1$ furthermore,
    the factor $[4]$ is essentially ${\rm ST}_4$.
    \medskip

\noindent    
{$(2)$} If $e_i-r$ or $e_j \in {\bf Z}$, then $S\!E_6$ factorizes as follows:
     For $i=1,\, \dots,\, 6$ and $j=7,\, 8,\, 9$,
$$\begin{array}{ccccccc}
    e_i-r,\  e_j =\ &\ \dots\ &\ -1\ &\ 0\ &\ 1\ &\ 2\ &\ \dots\\
    \ &\ \dots\ &\ [51]\ &\ [51]A0\ &\ [15]A0\ &\ [15]\ &\ \dots\end{array}$$
When $e_i-r,e_j=0,1$, the factor $[5]$ is essentially $S\!E_5$.
\medskip

\noindent
 {$(3)$} If $r\in {\bf Z}$, then $S\!E_6$ factorizes as follows:
  $$ \begin{array}{ccccccccc}
    r=\ &\ \dots\ &\ -2\ &\ -1\ &\ 0\ &\ 1\ &\ 2\ &\ 3\ &\ \dots\\
    \ &\ \dots\ &\ [1113]\ &\ [1113]A0\ &\ [1131]A0\ &\ [1311]A0\ &\ [3111]A0\ &\ [3111]\ &\ \dots\end{array}$$
    When $r=-1,0,1,2$, the factor $[3]$ is essentially $S\!E_3$.
\end{prp}

\begin{proof} Recall Remark \ref{factorSameType}. 
(1) comes from Theorem \ref{SE6A17} and the fact that by taking adjoint $\Delta$ changes into $-\Delta+1$. Proposition \ref{E6red_e9} says if $e_9=0$ then it factors as [51] and the factor [5] is $E_5$; to get (2), we have only to recall that $S\!E_6$ is the restriction of $E_6$ under the condition:
$$e_1-2e_2+e_3=e_4-2e_5+e_6=e_7-2e_8+e_9,$$
and $S\!E_5$ is the restriction of $E_5$ under the same condition with $e_9=0$,
as is stated in \S 1.1. (3) is due to Proposition \ref{prp1113}.
\end{proof} 

\subsubsection{Shift relation when $S\!E_6$ is divisible by $E_2$}\label{RedSE6shift}

Theorem \ref{redSE6} says that when $\Delta=0$, $S\!E_6$ is divisible by the Gauss equation $E_2$ from the right, and when $\Delta=1$, from the left. In these cases, We see the factorization of shift relations, for example, when  $\Delta=-a$; for the remaining cases, we have the same type of factorization. We omit the proof.

In \S\ref{factor_shift_rel_SE3}, for $S\!E_3$, factorization of the shift relations are studied when $a=0,1$. They reduce to a shift relation of $E_2$ and a trivial identity. Here, for $S\!E_6$, we see that the factorization of the shift relations, for $a=0,1$,  reduce to a shift relation of $E_2$ and that of ${\rm ST}_4$.

\begin{thm}\label{SE6breakasE2E4}
Relative to the shift relations:
\[E(a)\circ P_{a+}(a-1)=Q_{a+}(a-1)\circ E(a-1)\]
  and
  \[E(a-1)\circ P_{a-}(a)=Q_{a-}(a)\circ E(a),\]
  where $E=S\!E_6$, assume $a=0$. Then, they factorize as
$$[4,2]\circ [1,4]=[4,1]\circ [2,4]\quad{\rm and}\quad [2,4]\circ [3,2]=[2,3]\circ [4,2],$$
and, up to multiplication of functions, left and right factors are canceled, and they reduce to
$$[2]\circ [1]=[1]\circ [2]\quad{\rm and}\quad [4]\circ [3]=[3]\circ [4],$$
which are a shift relation of $E_2$ and that of ${\rm ST}_4$, respectively. 
\end{thm}


\section{Review of fundamental properties of $E_5$}\label{E5}

We specialize  $E_6=T_0+T_1\partial+T_2\partial^2+T_3 \partial^3$ by the condition $e_9=0$, and define $E_5$ as in \S\ref{TheTenTable}:
\[E_5:=E_6(e_9=0)/\partial=x\overline{T}_0+\overline{T}_1+\overline{T}_2\partial +\overline{T}_3\partial ^2.\]
We review some properties of $E_5$ from \cite[\S 9]{HOSY1}.

  \subsection{Shift operators, S-values  and reducibility conditions  of $E_5$}\label{E5Shift}

  As in the case of equation $E_6$, equation $E_5$ has shift operators relative to the shifts
  of blocks ${\bm e}_1=\{e_1,e_2,e_3\}$ and  ${\bm e}_4=\{e_4,e_5,e_6\}$.
For example, if  
$P_{\pm0}$ denotes the shift operator of $E_5$
for the shift ${{\bm e}_1}\pm{\bm 1}$,
and $P_{0\pm}$ for  ${{\bm e}_4}\pm{\bm 1}$, then 
$$P_{-0}=(x-1)\partial +1-r,\quad P_{0-}=x\partial +1-r,\quad r = (e_1+\cdots+e_8)/3-2.$$

\begin{prp}\label{S-valueE5} The S-values for the shifts of blocks:
\[ \def\arraystretch{1.1} \begin{array}{ll}
 Sv_{-0}&=P_{+0}({\bm e}_1-1)\circ P_{-0}=(r-1)(r-2)(e_4-r)(e_5-r)(e_6-r),\\
 Sv_{0-}&= P_{0+}({\bm e}_4-1)\circ P_{0-}= -(r-1)(r-2)(e_1-r)(e_2-r)(e_3-r).
\end{array}\]
\end{prp}

\begin{thm} \label{redcondE5}
  If one of $r,\ e_i-r \ (1\le i\le6)$ is an integer,
  then the equation $E_5$ is reducible.
  \end{thm}

\subsection{Reducible cases of $E_5$}\label{E5Fact}
When $E_5$ is reducible as in Theorem \ref{redcondE5},
the equation $E_5$ factorizes and $E_4$ and $E_3$ appear as factors:

\begin{prp}\label{factorE5} $(1)$ If $e_i-r \in {\bf Z}$, then
  $E_5$ factorizes as follows: For $i=1,\dots,6$, 
  $$\begin{array}{ccccccc}
    e_i-r= \ &\ \cdots\ &\ -1\ &\ 0      \ &\     1  \ &\ 2\ &\ \cdots\\
    \ &\  \cdots\ &\ [4,1] \ &\ [4,1]A0\ &\ [1,4]A0\ &\ [1,4] \ &\ \cdots
  \end{array}$$
  When $e_i-r=0,1$, the factor $[4]$ is {essentially} $E_4$
  defined in Section $\ref{E4ap1}$.
    \par\noindent
$(2)$ For the case $r\in {\bf Z}$:   $$\begin{array}{cccccccc}
    r=\ &\  \cdots\ &\ -1    \ &\  0\ &\      1\ &\ 2\ &\ 3\ &\ \cdots\\
    \ &\ \cdots\ &\ [1,1,3] \ &\ [1,1,3]A0\ &\ [1,3,1]A0\ &\ [3,1,1]A0\ &\ [3,1,1]
    \ &\ \cdots\end{array}$$
    When $r=0,1,2$, the factor $[3]$ is {essentially} $E_3$.
\end{prp}


\section{$S\!E_5(a,b,c,g,p,q)$: a specialization of $E_5$}\label{SE5}

In this section we introduce
 a codimension-2 specialization $S\!E_5$ of $E_5$.  
 As we saw in Section \ref{RedSE6}, the specialization $S\!E_6$
 has various properties that $E_6$ does not have.
 We imitate this specialization to get a new equation $S\!E_5$ from $E_5$.

\subsection{Definition of $S\!E_5$ as a specialization of $E_5$}\label{SE5Def}
The equation $S\!E_5$ is defined as the specialization of $E_5$,
characterized by the condition
$$e_1-2e_2+e_3=e_4-2e_5+e_6=e_7-2e_8,$$
which is derived by the characterization of $S\!E_6$ by assuming
$e_9=0$. Hence, the parameter $r$ used for $S\!E_6$ is written as
\begin{equation}\label{rabpq}
  r=-2a-2b-2c-g-p-q-3.
\end{equation}
Substituting this identity to the expressions of $e_i$ for $S\!E_6$
given in \S 4.1, we see that $S\!E_5$ is parameterized by $\{a,b,c,g,p,q\}$ as
\begin{equation}\label{parse5}
  \def\arraystretch{1.1}
  \begin{array}{ll}
  e_1 = -2a - 2b - 2c - g - q - 2,\quad 
  &e_2 = -a - c - 2b - g - q - 1, \\
  e_3 = -2b - q,
  & e_4 = -2a - 2b - 2c - g - p - 2, \\
  e_5 = -b - c - 2a - g - p - 1,
  &e_6 = -2a - p,\\
  e_7 = 2a + 2b + g + 2, 
  &e_8 = a + b + 1.
  \end{array}
\end{equation}  

\subsection{From the shift operators of $S\!E_3$ to those of $S\!E_5$}

As is the case for $S\!E_6$ in \S\ref{shiftop_SE3_to_SE6},
we show how is $S\!E_5$ connected
by addition and middle convolution with $S\!E_3$. 

\subsubsection{From $S\!E_5$ to $S\!E_3$}
  \begin{enumerate}
  \item In the $(\theta,\partial)$-form of $\partial S\!E_5$, replace $\theta$ by $\theta+r$. The resulting equation is divisible from the left by $\partial^3$.
  \item Divide by $\partial^3$.
  \item The local exponents at 0 and 1 of the resulting equation are
$$p,\ p + a + c + 1,\ 2 a + 2 c + g + p + 2;\quad
  q,\ q + b + c + 1,\ 2 b + 2 c + g + q + 2.$$
Multiply $X^{-1}$ from the left and $X$ from the right, where
$ X:=x^p(x-1)^q$.
Then we get $S\!E_3$:
\end{enumerate}

\begin{prp}
  \begin{equation}\label{SSE3_SE5}
S\!E_3=X^{-1}\cdot\partial^{-3}\cdot\partial^{r}\cdot S\!E_5\cdot \partial^{-r}\cdot X,\quad X=x^p(x-1)^q.  \end{equation}
\end{prp}

In other words, given $u_3$ such that $S\!E_3u_3=0$, define $u_5$ as
$u_5:=\partial^{-r}Xu_3;$
then $S\!E_5u_5=0$.

\subsubsection{Recipe for computing shift operators of $S\!E_5$}
\label{recipe_3_to_5}

In view of the relation $(\ref{SSE3_SE6})$ the shift operators
for the shifts of $(a,b,c,g)$ can be derived
from the corresponding shift operators of $S\!E_3$.

In fact, we can follow a procedure similar to
\S \ref{recipe_3_to_6}:

For a shift $sh:(\mathbf{a})\to sh(\mathbf{a})$ of $\mathbf{a}=(a,b,c,g)$, let $S\!E_3(sh(\mathbf{a}))$ and $S\!E_5(sh(\mathbf{a}))$   be the shifted operators of $S\!E_3(\mathbf{a})$ and $S\!E_5(\mathbf{a})$, respectively.
\begin{itemize}
\item 
 Given $v_3$ such that $S\!E_3(sh(\mathbf{a}))v_3=0$, define $v_5$ as $v_5:=\partial^{-r}Xv_3;$
then $S\!E_5(sh(\mathbf{a}))v_5=0$, and conversely given $v_5$ such that $S\!E_5(sh(\mathbf{a}))v_5=0$, define $v_3$ as
$v_3:=X^{-1}\partial^{r}v_5;$
then $S\!E_3(sh(\mathbf{a}))v_3=0$.
\item Let $P=P_3$ be the shift operator for the shift $sh$ of $S\!E_3$ found in  \S 3.
$$P:{\rm Sol}(S\!E_3(\mathbf{a}))\to{\rm Sol}(S\!E_3(sh(\mathbf{a}))).$$ 
When $u_3$ solves $S\!E_3((\mathbf{a}))u_3=0$, 
$$v_3:=Pu_3\quad{\rm solves}\quad S\!E_3(sh(\mathbf{a}))v_3=0,$$ 
turning to $E_5$, 
$$v_5:= \partial^{-r}\cdot Xv_3\quad{\rm solves}\quad S\!E_5(sh(\mathbf{a}))v_5=0.$$ 
Rewriting $v_3=Pu_3$, by $v_3=X^{-1}\partial^{r}v_5$ and  $u_3=X^{-1}\partial^{r}u_5$. we have
\begin{equation}\label{mc_5}\partial^{r}v_5=X\cdot P\cdot X^{-1}\cdot\partial^{r}u_5.
\end{equation}

\item Apply middle convolution as in \S \ref{recipe_3_to_6}, 
we transform  \eqref{mc_5} to an identity like
$$R_2\ v_5 = R_1 \ u_5.$$ 

\item
Apply Euclidean algorithm to $\mathbf{E}=S\!E_5(sh(\mathbf{a}))$ and $R_2$.
\end{itemize}

Thus, we have the following theorem.

\begin{thm}\label{shiftopSE5}
   The equation $S\!E_5$ admits shift operators for the shifts $a\to a\pm1$,
   $b\to b\pm1$, $c\to c\pm1$, $g\to g\pm2$, $p\to p\pm1$, and
   $q\to q\pm1$.
\end{thm}

\subsection{Shift operators, S-values and reducibility conditions of $S\!E_5$}\label{SE5Shift}

The shift operators of $S\!E_5$ obtained
for the shifts of $a,b,c,g$ are less complicated than those of $S\!E_6$;
however the operators appearing in the process are still very complicated.
In the following, we give a sketch on how to solve directly the shift equation $(EPQE)$ to get shorter expressions.

\subsubsection{Shift operators coming from those of $E_5$}\label{SE5ShiftfromE5}
Due to the relation $(\ref{parse5})$, 
the shifts $p\to p\pm1$ and $q\to q\pm1$ cause the shifts of the blocks
${\bm e}_4 =\{e_4, e_5, e_6\} \to {\bm e}_4 \mp {\bm 1}
=\{e_1\mp 1, e_2\mp 1, e_3\mp 1\}$
and ${\bm e}_1=\{e_1, e_2, e_3\} \to {\bm e}_1\mp {\bm 1}
 =\{e_1\mp 1, e_2\mp 1, e_3\mp 1\}$
 of the exponents of $E_5$, respectively,
 their shift operators come from those of $E_5$.
 For example,
 \[\def\arraystretch{1.1}\setlength\arraycolsep{3pt}
 \begin{array}{lcl}
P_{p+}&=&x\partial+2a+2b+2c+g+q+p+4=P_{0-}:=x\partial +1-r, \\
P_{q+} &=&(x-1)\partial+2a+2b+2c+g+q+p+4=P_{-0}:=(x-1)\partial +1-r,
 \end{array}\]
where $P_{0-}$ and $P_{-0}$ are defined for $E_5$ in \S\ref{E5Shift}.
We refer to \cite[\S 9.3]{HOSY1} for the inverse
operators $P_{p-}$ and $P_{q-}$.
The S-values $Sv_p$ and $Sv_q$ for the shifts $p\to p+1\to p$
and $q\to q+1\to q$ are listed in \S \ref{SE5ShiftS}.
  
\subsubsection{Shift operators for $a\rightarrow a+1$}\label{SE5Shiftap}

 For the shift $a\rightarrow a+1$, the procedure in
\S \ref{secshift} can be applied with a slight modification as follows:

Due to the correspondence above, the shift
 $\{a \to a+1\}$ in terms of $e_i$ is expressed as
 \[ \begin{array}{ll}sh_{a+}&:=\{e_1\to e_1-2,\ e_2\to e_2-1,
   \ e_4\to e_4-2,\ e_5\to e_5-2,\    e_6\to e_6-2,\\
   &\qquad  e_7\to e_7+2,\ e_8\to e_8+1,\ r\to r-2\}.
 \end{array}\]
 The shifted equation is written as
 $\overline T_s =x\overline T_{0s}+\overline T_{1s}+\overline T_{2s}\partial+\overline T_{3s}\partial^2$, where
\[\def\arraystretch{1.1}\setlength\arraycolsep{3pt}
\begin{array}{lcl}
\overline T_{0s} &=& (\theta -r+3)(\theta -r+4)(\theta -r+5)(\theta +e_7+3)(\theta +e_8+2), \\
\overline T_{1s} &=& (\theta -r+3)(\theta -r+4)B_{51s}, \\
\overline T_{2s} &=& (\theta -r+4)B_{52s}, \\
\overline T_{3s} &=& -(\theta +5-e_1)(\theta +4-e_2)(\theta +3-e_3),
\end{array}\]
and
\[ B_{51s}= {\rm subs}(a=a+1,B_{51}),\qquad B_{52s} ={\rm subs}(a=a+1,B_{52}).\]

  We assume the shift operators are of the form:
\begin{eqnarray*}
    P&=&x^2P_{nn}+xP_{n}+P_0+P_1\partial+P_2\partial^2, \quad Q=x^2Q_{nn}+\cdots+Q_2\partial^2,
\end{eqnarray*}
  as before in \eqref{PQz},
  while the numbering of the components of
  differential equations are different from those used in
  Lemma \ref{syseq}, we need to modify the
  formulas in Lemma \ref{syseq} a little. 
  The equation
  \[ \overline T_s\circ P=Q\circ\overline  T,\]
  is equivalent to the following system of equations:
  \[ E3=0, \quad E2=0,\quad \dots, \quad Ed4=0,\]
  where
\[\def\arraystretch{1.2}\setlength\arraycolsep{3pt}
\begin{array}{rcl}
  E3 &=& \overline T_{0s}[2]P_{nn} - Q_{nn}[1]\overline T_0, \\
E2 &=& \overline T_{0s}[1]P_{n} + \overline T_{1s}[2]P_{nn}  - Q_{nn}\overline T_1 - Q_{n}[1]\overline T_0, \\
E1 &=& \overline T_{0s}P_0 + \overline T_{1s}[1]P_{n} + (\theta +2)\overline T_{2s}[1]P_{nn}  \\
&&\qquad    - \theta Q_{nn}[-1]\overline T_2[-1] - Q_{n}\overline T_1 - Q_0[1]\overline T_0, \\
E0 &=& \theta \overline T_{0s}[-1]P_1[-1] + \overline T_{1s}P_0 + (\theta +1)\overline T_{2s}P_{n} + (\theta +1)(\theta +2)\overline T_{3s}P_{nn} \\
&&\qquad   - \theta (\theta -1)Q_{nn}[-2]\overline T_3[-2] - \theta Q_{n}[-1]\overline T_2[-1] - Q_0\overline T_1 - Q_1(\theta +1)\overline T_0,\\
Ed1 &=& \theta \overline T_{0s}[-1]P_2[-1] + \overline T_{1s}P_1 + \overline T_{2s}P_0[1] + (\theta +2)\overline T_{3s}P_{n}[1]\\
&&\qquad   - \theta Q_{n}[-1]\overline T_3[-1] - Q_0\overline T_2 - Q_1\overline T_1[1] - (\theta +2)Q_2\overline T_0[1], \\
Ed2 &=& \overline T_{1s}P_2 + \overline T_{2s}P_1[1] + \overline T_{3s}P_0[2] - Q_0\overline T_3 - Q_1\overline T_2[1] - Q_2\overline T_1[2], \\
Ed3 &=& \overline T_{2s}P_2[1] + \overline T_{3s}P_1[2] - Q_1\overline T_3[1] - Q_2\overline T_2[2], \\
Ed4 &=& \overline T_{3s}P_2[2] - Q_2\overline T_3[2].
\end{array}\]
We solve these equations applying to $\overline T=S\!E_5$
and the shifted equation $\overline T_s$,
following a procedure similar to that in \S \ref{SE6Shifta} and we get:

\begin{eqnarray*}
P_{nn}&=&Q_{nn}=P_{n}=Q_{n}=0, \\
P_0 &=& (\theta -r+1)(\theta -r+2)(\theta +e_7+1)(\theta +e_8+1), \\
Q_0 &=& (\theta -r+3)(\theta -r+4)(\theta +e_7+2)(\theta +e_8+1), \\
P_1 &=& (\theta -r+2)P_{13}, \quad Q_1 = (\theta -r+4)Q_{13}, \\
P_{13}&=& -2\theta ^2-(4a+7b+c+2g+2q+12)\theta  -(6ab+3aq+6b^2+2bc+3bg+3bq\\
&&+gq+8a+20b+2c+4g+6q+16), \\
Q_{13} &=& P_{13}-2\theta -(5+a+3b+q), \\
P_2 &=& (\theta +3-e_2)(\theta +3-e_3), \quad Q_2 = (\theta +4-e_2)(\theta +3-e_3).
\end{eqnarray*}

\subsubsection{Shift operators for $a\to a-1$}\label{shiftopPam_ofSE5}

Since we know explicitly the shift operator $P_{a+}$
for the shift $a\to a+1$,
to find $P_{a-}$, we solve the equation
  \[ P_{a-}(a=a+1)\circ P_{a+} =D\circ S\!E_5+ Sv_{a+},\]
for some differential operator $D$ and a constant $Sv_{a+}$ 
(depending on the parameters), which will be the S-value.
\bigskip

Computation is done by assuming
that $P_{a-}$ is expressed in $(x,\partial)$-form as
\[P_{a-}=x^3(x-1)^3R_2(x)\partial^4 + x^2(x-1)^2R_3(x)\partial^3
+x(x-1)R_4(x)\partial^2
+ R_5(x)\partial+R_4^1(x),
\]
where $R_k(x)$ is a polynomial of degree $k$ in $x$, and $R_4^1(x)$ of
degree $4$. Then, we get
the polynomial $R_2(x)$ written as
  \[ \begin{array}{rcl}  
    R_2(x)& =&  -(e_7-1)Nx^2 +(e_7-1)(e_2-r-1)(2e_3-3r-3)x \\
          && \qquad +(e_2-r-1)(e_3-r-1)(e_3-r-2), \\
    N & =& -2a^2-3ag-2ab+aq+ap+2+3p-g+4c+q+2c^2-g^2\\
    && \qquad +p^2-bg+cq+pq-cg+3cp.
  \end{array}\]
  The other terms are too long to be printed here.

The $S$-value is seen to be 
\begin{eqnarray*}
  Sv_{a+} &=&  (a+1)(2a+g+2)(a+c+p+2)(2a+2b+g+3)\\
  &&\ \times(2a+2c+g+p+4)(2a+2c+g+p+3)(2a+2b+2c+g+4)\\
  &&\ \times(a+b+c+g+2)(1-r)(2-r)(3-r).
\end{eqnarray*}
The operator $Q_{a-}$ obtained by solving 
$S\!E_5(a=a-1)\circ P_{a-}= Q\circ S\!E_5$ is of the same type as $P_{a-}$.

We note here that $P_{a-}$ is expressed in $(x,\theta ,\partial)-$form as
$$x^4P_{nnnn}+x^3P_{nnn}+x^2P_{nn}+xP_n+P_0+P_1\partial,$$
which is very different from the form of
  $P_{a+}$ for $S\!E_6$. The top and the tail coefficients are given as
\[\def\arraystretch{1.1}
\begin{array}{rcl}
  P_{nnnn}&=& -(e_7-1)(\theta +e_8+1)(\theta +e_7+1)(\theta -r+2)(\theta -r+1)N,\\
  P_1 &=& -(e_2-r-1)(e_3-r-1)(e_3-r-2)(\theta +2-e_1)\\
  && \qquad \times (\theta +2-e_2)(\theta +2-e_3),
\end{array}\]
while the other terms are too long.

\subsubsection{Shift operators for $b\to b\pm 1$
  and other shifts}\label{SE5Shiftb}

The shift operators for the shifts $b\to b+1$, $c\to c + 1$,
and $g\to g+ 2$ are obtained similarly to the previous case $a\to a+1$.
The type of these three operators is $P_0 + P_1\partial
  +P_2\partial^2$, the same as $P_{a+}$.

By computing the operator $P_{b-}$ for the shift $b\to b-1$
similarly to the case $a\to a-1$, 
we get S-value for $\{b\to b+1\to b\}$:
\[\def\arraystretch{1.1}\setlength\arraycolsep{3pt}
\begin{array}{lcl}
  Sv_{b+} &=& (b+1)(2b+g+2)(b+c+q+2)(q+g+4+2c+2b)\\
  &&\  \times (q+g+3+2c+2b)(2a+2b+g+3)(g+2+c+b+a)\\
  &&\  \times (g+4+2c+2b+2a)(2a+2b+2c+p+q+g+4)\\
  &&\  \times (2a+2b+2c+p+q+g+5)(2a+2b+2c+p+q+g+6).
\end{array}\]

To get the operators for the shifts $c\to c-1$ and $g\to g-2$, as we did for $S\!E_6$ in \S\S 5.2.2 and 5.2.3, we make a detour, a composition
  of certain shifts: 
\[\def\arraystretch{1.2}\begin{array}{cl}
  c\to c-1:&(c,p,q)\underset{P_{cpq}}\Rightarrow(c-1,p+1,q+1)
  \underset{P_{p-}(p+1)}\to(c,p-1,q-1)\\ 
  &\underset{P_{q-}(q+1)}\to(c-1,p,q),\\ 
g\to g-2:&(a,g)\underset{P_{ag}}\Rightarrow(a+1,g-2)\underset{P_{a-}(a+1,g-2)}\to(a,g-2),
\end{array}\]
where the operators for $\Rightarrow$ are to be found.
The procedure is almost the same as for $S\!E6$; we give some summaries:

Let $P_{cpq}$ and $P_{rcpq}$ be the operators for the shifts
$$(c,p,q)\to (c-1, p+ 1, q+ 1) \quad {\rm and}\quad (c,p,q)\to (c+1,p-1,q-1),
$$
respectively.
Their $(x,\partial)$-forms are of the type
\[x^3(x-1)^3A\partial^4+x^2(x-1)^2R_3\partial^3+x(x-1)R_4\partial^2
+R_5\partial+R^1_4,\]
where $A$ is a constant $(2a+2b+4c+2g+p+q+4)$ for $P_{cpq}$ and $(p+q)$ for
$P_{rcpq}$, and $R_k$ denotes polynomial of degree $k$ of $x$, $R_4^1$ of 4.
Their $(x,\theta ,\partial)$-forms are 
$x^2P_{nn}+ \cdots + P_1\partial$. 

Let $P_{ag}$ and $P_{rag}$ be the operators for the shifts
$$(a,g)\to (a-1, g+ 2) \quad {\rm and}\quad (a,g)\to (a+1,g-2),
$$
respectively.
Their $(x,\partial)$-forms are of the type
\[x^3(x-1)^2R_1\partial^4+x^2(x-1)R_2\partial^3+xR_3\partial^2
+R^1_3\partial+R^1_2,\]
where $R_k$ and $R_k^1$ denote polynomials of degree $k$ of $x$.
Their $(x,\theta ,\partial)$-forms are
$x^2P_{nn}+ \cdots + P_1\partial$. 
Thus, we can now compute the corresponding
S-values $Sv_{cpq}$ for the shift $$(c,p,q)\to (c-1, p+ 1, q+1)\to (c,p,q) $$ and
$Sv_{ag}$ for the shift $$(a,g)\to (a+1,g-2)\to (a,g)$$ are as follows:
\[\def\arraystretch{1.1}\setlength\arraycolsep{2pt} 
\begin{array}{rcl}
Sv_{cpq}&=& c(p+1)(q+1)(2c+g)(2b+2c+g+q+2)(2a+2c+g+p+2)\\
&&\ \times (2a+2b+2c+g+2)(a+b+c+g+1)\\
&&\ \times (2a+2b+2c+p+q+g+4)^2, \\
Sv_{ag} &=& (g-1)(a+1)(a+c+p+2)(a+b+c+g+1)(2c+g)(2b+g)\\
&&\ \times (q+2+g+2c+2b)(q+1+g+2c+2b)\\
&&\ \times (2a+2b+2c+p+q+g+4)^2.
\end{array}\]
\setlength\arraycolsep{3pt}  

We can next compute the operator $P_{c-}$  by composing
the operators $P_{cpq}$, $P_{p-}$, and $P_{q-}$; and $P_{g-}$ by composing 
$P_{ag}$ and $P_{a-}$  as we did for $S\!E_6$ and we get
the S-values associated for $P_{c-}$ and $P_{g-}$.

\begin{remark} The explicit expressions of the shift operators
  are listed in the file named {\it SE5PQ.txt} in the site
  given in Introduction.
\end{remark}

\subsubsection{S-values of the equation $S\!E_5$ and the reducibility}\label{SE5ShiftS}
We here list the S-values obtained above
partially using  $e_1,\dots, e_6,r$ for simplicity.

\begin{prp}\label{S-valueSE5}
  The S-values of the shift operators for $a,b,c,g,p,q$:
  \[\def\arraystretch{1.1}\setlength\arraycolsep{2pt} 
  \begin{array}{rcl}  
Sv_{p+} &=& (e_1-r)(e_2-r)(e_3-r)(1-r)(2-r), \\
Sv_{q+}&=& (e_4-r)(e_5-r)(e_6-r)(1-r)(2-r), \\
Sv_{a+}&=&  (a+1)(2a+g+2)(2a+2b+g+3)(a+b+c+g+2)(1-r)(2-r)\\
&&\ \times (2a+2b+2c+g+4)(3-r)(e_2-r)(e_3-r)(e_3-r+1), \\
Sv_{b+}&=&  (b+1)(2b+g+2)(2a+2b+g+3)(a+b+c+g+2)(1-r)(2-r)\\
 &&\ \times (2a+2b+2c+g+4)(3-r)(e_5-r)(e_6-r)(e_6-r+1), \\
Sv_{cpq}&=&  c(p+1)(q+1)(2c+g)(a+b+c+g+1)(2a+2b+2c+g+2) \\
 &&\ \times (e_3-r-1)(e_6-r-1)(1-r)^2, \\
Sv_{ag}&=&  (a+1)(g-1)(2c+g)(2b+g)(a+b+c+g+1) \\
&&\ \times  (e_6-r-1)(e_6-r-2)(e_2-r)(1-r)^2, \\
Sv_{c-}&=& (p+1)(q+1)(c+1)(2c+g+2)(2a+2b+2c+g+4)\\
&&\ \times (a+b+c+g+2)(e_2-r)(e_3-r+1)(e_3-r)(e_5-r)\\
&&\ \times (e_6-r+1)(e_6-r)(1-r)(2-r)^2(3-r)^2,\\
Sv_{g-}&=&(g+1)(2a+g+2)(2b+g+2)(2c+g+2)(2a+2b+g+3)\\
&&\ \times (a+b+c+g+3)(a+b+c+g+2)(2a+2b+2c+g+4)\\
&&\ \times (e_3-r+1)(e_3-r)(e_6-r+1)(e_6-r)(1-r)(2-r)(3-r).
  \end{array}\]
  \setlength\arraycolsep{3pt}  
\end{prp}
We have seen that $E_5$ is reducible when one of
$r,\ e_1-r,\ \cdots,\ e_6-r$ is an integer (Theorem \ref{redcondE5}).
In terms of the present parameters for $S\!E_5$, the values are
\[
\begin{array}{ll} \label{5eredinab}
  r=-2a-2b-2c-g-p-q-3,\\
  e_1-r = p+1, \quad e_2-r = a+c+p+2, & e_3-r=2a+2c+p+g+3,\\
  e_4-r = q+1, \quad e_5-r = b+c+q+2, & e_6-r=2b+2c+q+g+3.
\end{array}
\]
Thus we have the following conditions of reducibility of the
equation $S\!E_5$.

\begin{thm} \label{redcondSE5}
  If one of the values
\begin{gather*}
-a,\ -b,\ -c,\ 1/2-g/2,\ a+g/2+1,\ b+g/2+1,\ c+g/2+1, \\[2mm]
\delta+g+2,\ -(\delta+g/2+1),\quad  a+b+g/2+3/2, \\[2mm] 
e_i-r\ (1\le i\le 6),\ r
\end{gather*}
is an integer, then the equation $S\!E_5$ is reducible.
\end{thm}

\subsection{Reducible cases of $S\!E_5$}\label{SE5Red}
Paraphrasing Proposition \ref{factorE5}, by specializing the parameters,
we have

\begin{prp}\label{SE5Redprp1} For $i=1,\dots,6$,
  \[ \def\arraystretch{1.1}
\begin{array}{rccccccc}
 e_i-r=\  &\cdots&  -1,&0,   &1,    &2,    &\cdots\\
 &\cdots&[4,1]&[4,1]A0&[1,4]A0&[1,4]&\cdots\\
 r=\ &\cdots&  -2,&-1,       &0,       &1,       &2,      &\cdots\\
 &\cdots&[1,1,3]&[1,1,3]A0&[1,3,1]A0&[3,1,1]A0&[3,1,1]&\cdots
\end{array}\]
When the factors are free of apparent singularities,
the factor $[4]$ is essentially a specialization $S\!E_4$ of $E_4$, defined in \S $1.1$ (cf. \S $\ref{newSE4}$),
and the factor $[3]$ is essentially $S\!E_3$, defined in \S $1.1$.
\end{prp}

\begin{prp}\label{factorSE5} Let $\Delta$ be one of $\{-a, -b, c+g/2+1, -(\delta+g/2+1)\}$, and $a+b+g/2+3/2$. Then
  \[ \def\arraystretch{1.1}
 \begin{array}{rcccccc}
   \Delta=&\cdots&  -1,&0,   &1,    &2,    &\cdots\\
    &\cdots\ &\ [4,1]\ &\ [4,1]A0\ &\ [1,4]A0\ &\ [1,4]\ &\ \cdots\\
   \end{array}\]
 When the factors are free of apparent singularities,
 the factor $[4]$ is essentially equal to a specialization of ${\rm ST}_4$, defined in \S $1.1$; ${}_4E_3$ only for the last case $a+b+g/2+3/2=0,1$.
\end{prp}

\begin{proof} (Sketch)
When $a=-1$,  $S\!E_5$ factors as [4,1], the factor $[1]=\partial+(2b+2c+g+q+2)/x+p/(x-1)$, and [4] has local exponents
\[ \def\arraystretch{1.1} \begin{array}{cccccccccc}
0,&\ 1,&\ -2b-q-1,&\ -1-c-2b-g-q;\\ 
0,&\ 1,&\ -b-c-g-p,&\ -1-p-2b-2c-g;\\ 
s,&\ s+1,&\ 1+b,&\ 1+2b+g,
\end{array}\]
where $s=2 + 2b + q + 2c + g + p$.
We see that it has the spectral type $[211,\ 211,\ 211]$,
which shows it is a specialization of ${\rm ST}_4$ with five free exponents.
\par
When $g=-2a-2b-3$,  $S\!E_5$ factors as [4,1],
  the factor $[1]=\partial$, and [4] has local exponents
\[ \def\arraystretch{1.1} \begin{array}{ccccccc}
  0,&\  -1-2c-q,&\  -c+a-q,&\  -2b-q-2;\\ 
0,&\  -p-1-2c,&\  -p+b-c,&\  -2a-2-p;\\ 
2+2c+p+q,&\ 3+2c+p+q,&\ 4+2c+p+q,&\ a+b+3.
\end{array}\]
Exchanging $x=1$ and $\infty$, and multiplying $(x-1)^{2+2c+p+q}$
from the right, we see that it has the spectral type $[1111,\  31,\  1111]$,
which shows it is a specialization of ${}_4E_3$ with five free exponents.
\end{proof} 

\begin{prp}\label{factorSE5cont32}Let $\Delta$ be one of $\{-c, 1/2-g/2, a+g/2+1, b+g/2+1, a+b+c+g+2\}$. Then
\[ \def\arraystretch{1.1} \begin{array}{rcccccc} 
 \Delta=&\cdots&  -1,&0,   &1,    &2,    &\cdots\\
    &\cdots\ &\ [3,2]\ &\ [3,2]A0\ &\ [2,3]A0\ &\ [2,3]\ &\ \cdots\\
\end{array}\]
When the factors are free of apparent singularities,
the factor $[3]$ is essentially ${}_3E_2$,
after exchanging $x=1$ and $\infty$. Refer to \S $\ref{SE5divE2}$.
\end{prp}

\subsubsection{When $S\!E_5$ is divisible by $E_2$ from the right, proof of Proposition \ref{factorSE5cont32}}\label{SE5divE2}  

Consider when $S\!E_5$ is divisible by Gauss equation $E_2$ on the right.
Since
  \[ S\!E_5=x\overline T_{0} + \overline T_{1} + \overline T_{2}\partial + \overline T_{3}\partial^2,\]
where $\overline T_{0}$, $\overline T_{1}$, $\overline T_{2}$ and $\overline T_{3}$ are 
given in \S\ref{SE5Shift} and Gauss equation is 
  \[ E_2 =E_0 + E_1\partial;\qquad   E_0=(\theta +A)(\theta +B), \quad E_1=-(\theta +C),\]
as in \S\ref{EarGauss}, when $S\!E_5$ is divisible by $E_2$,
we have an equation
  \[ S\!E_5=  Q\circ E_2\]
for some operator
  \[Q=xQ_0 + Q_1 + Q_2\partial,\]
where $Q_i$ are polynomials of $\theta $.
In this case, by the identity
 \[ \begin{array}{l}   
  Q\circ E_2=xQ_0E_0+Q_0(\theta -1)E_1(\theta -1)\theta +Q_1E_0\\
  \hskip60pt +Q_1E_1\partial+Q_2E_0(\theta +1)\partial
  +Q_2E_1(\theta +1)\partial^2,
  \end{array}\]
we see that
\begin{eqnarray*}
 -Q_2\times (\theta +1+C)   &=& -(\theta +3-e_1)(\theta +3-e_2)(\theta +3-e_3), \\
 Q_0\times (\theta +A)(\theta +B) &=& (\theta +1-r)(\theta +2-r)(\theta +3-r)(\theta +e_7+1)(\theta +e_8+1);
\end{eqnarray*}
Thus, the possibility of $A$, $B$ and $C$ is
\[
\begin{array}{rcl}  
  A,\quad B &=& 1-r, \ 2-r, \ 3-r, \ e_7+1, \ e_8+1, \\
    C    &=&  2-e_1, \ 2-e_2, \ 2-e_3.
\end{array}
\]
For each possible choice of $A$, $B$, and $C$, we check
whether $S\!E_5$ is divisible by $E_2$ (by computer assistance),
and we find the following cases:
\[
\begin{array}{lll}
\text{case} & \text{condition} & E_2(A,B,C) \\ \hline
  \noalign{\vskip2pt}
C1 & c=0     & A=1-r, \ B=e_7+1, \ C=2-e_1, \\
C2 & g=-2b-2 & A=1-r, \ B=e_7+1, \ C=2-e_1, \\
C3 & g=-2a-2 & A=1-r, \  B=e_7+1, \ C=2-e_3, \\
C4 & g=-a-b-c-2\quad & A=1-r,\  B=e_7+1, \ C=2-e_3, \\
C5 & g=1 & A=1-r, \ B=e_8+1, \ C=2-e_2. \\
  \noalign{\vskip2pt} \hline   
\end{array}
\]
Remark that there may be other choices of parameters $\{a,b,c,g,p,q\}$,
we here list the condition that is represented
as one linear form of parameters called a codimension-one condition.

Recall that the set of parameters 
$\{e_1, e_2, e_3, e_4, e_5, e_6, e_7, e_8, r\}$
and that of $\{a, b, c,\quad g, p, q\}$
are related as in \eqref{rabpq} and \eqref{parse5} in \S\ref{SE5Shift}.

We see that the quotient is written by the generalized hypergeometric equation ${}_3E_2$;
to give identities concretely, we modify this equation by changing
the coordinate $x$ to $1/(1-x)$ so that  ${}_3E_2$ is rewritten as
\[
\begin{array}{lcl}
GH_3&=&x^3(x-1)^2\partial^3+x^2(x-1)((4-a_0-a_1-a_2+b_1+b_2)x\\
 &&  -3+a_0+a_1+a_2)\partial+(2(1-a_0-a_1-a_2+b_1+b_2)x^2 \\
 &&\quad    +(b_1b_2-2b_2-2b_1-a_0a_1-a_1a_2-a_2a_0+3a_0+3a_1+3a_2-3)x \\
 &&\quad    +1+a_0a_1+a_1a_2+a_2a_0-a_0-a_1-a_2)\partial\\
 &&  -a_0a_1a_2.
\end{array}
\]

Then, we have the following identities:
\[
\begin{array}{lll}
\text{case}  &  Q=GH_3(a_0,a_1,a_2,b_1,b_2) \\ \hline
    \noalign{\vskip2pt}
C1   & a_0=a+p+2, a_1=2a+g+p+3, a_2=2a+2b+g+p+q+5,\\ &
         b_1=2a+p+3, b_2=2a+b+g+p+4 \\
C2   & a_0=a+c+p+2, a_1=2a+2c-2b+p+1, a_2=2a+2c+p+q+3,\\ &
         b_1=2a+2c+p+3, b_2=2a-b+c+p+2 \\
C3   & a_0=p+1, a_1=a+c+p+2, a_2=2b+2c+p+q+3,\\ &
        b_1=2a+p+3, b_2=b+c+p+2 \\
C4   & a_0=p+1, a_1=a+c+p+2, a_2=a+b+c+p+q+3,\\ &
        b_1=a+p+2, b_2=a+b+c+p+3 \\
C5   & a_0=p+1, a_1=2a+2c+p+4, a_2=2a+2b+2c+p+q+6,\\ &
        b_1=2a+p+3, b_2=2a+2b+2c+p+6. \\
  \noalign{\vskip2pt}\hline
\end{array}
\]


\section{Review of fundamental properties of $E_4$}
\label{E4ap1}

The equation
$$E_4=E_4(e_1,\dots,e_7)
  =\mathcal T_0+\mathcal T_1\partial+\mathcal T_2\partial^2$$
is defined as the equation with the Riemann scheme
$$R_4:\left(\begin{array}{llcll}x=0:&0&1&e_1&e_2\\ x=1:&0&1&e_3&e_4\\ x=\infty:&e_8&e_5&e_6&e_7\end{array}\right),\quad e_1+\cdots+e_7+e_8=4,
$$
where
\[\def\arraystretch{1.1}
\begin{array}{rcl}
    \mathcal T_0 &=& (\theta +e_5)(\theta +e_6)(\theta +e_7)(\theta +e_8),\\
    \mathcal T_1 &=& -2\theta ^3+\mathcal T_{12}\theta ^2+\mathcal T_{11}\theta +\mathcal T_{10},\\
    \mathcal T_{12} &=& s_{11}-s_{13}-5,\\
    \mathcal T_{11} &=& 3s_{11}-s_{21}+s_{22}-s_{23}-8,\\
    \mathcal T_2 &=& (\theta -e_1+2)(\theta -e_2+2),
\end{array}
\]
and $\mathcal T_{10}$ is the accessory parameter given by
\[\def\arraystretch{1.2} \begin{array}{rcl}
  54\mathcal T_{10}&=&-4s_{11}^3+ 4s_{12}^3
  - (6s_{12} - 27)s_{11}^2+(6s_{11}- 27)s_{12}^2\\
&& \ + (9s_{21} - 18s_{22} + 9s_{23} + 36)s_{11}
 - (9s_{22}-18s_{21}  + 9s_{23} - 18)s_{12}\\
&&\  - 81s_{21} + 81s_{22} - 27s_{23} - 27s_{3}- 135, 
\end{array}\]
where
\begin{gather*} 
s_{11}=e_1+e_2,\  s_{12}=e_3+e_4,\  s_{13}=e_5+e_6+e_7+e_8, \\
s_{21}=e_1e_2,\ s_{22}=e_3e_4,\  s_{23}=e_5e_6+e_5e_7+e_5e_8+e_6e_7+e_6e_8+e_7e_8,\\
s_3=e_6e_7e_8+e_5e_7e_8+e_5e_6e_8+e_5e_6e_7.
\end{gather*}
Note that these $s_*$ are a bit different from \eqref{e19tos}.
The equation $E_4$ has $(x,\partial)$-form as
$$E_4=x^2(x-1)^2\partial^4+p_3\partial^3+p_2\partial^2
+p_1\partial+p_0,\quad p_0=e_5e_6e_7e_8.$$

\subsection{A shift operator of $E_4$}

It is easy to check that 
\[E_4(e') \circ \partial = \partial\circ E_4(e),
\quad e'=(e_1-1,\dots,e_4-1,e_5+1,\dots,e_8+1),
\]
which, in particular, implies $E_4$ has differentiation symmetry.
Thus, $\partial$ is the shift operator for the shift $e\to e'$. 
Set
$$R=x^2(x-1)^2\partial^3+p_3\partial^2+p_2\partial+p_1.$$
Then we have
  \[R\circ \partial = E_4-p_0 \equiv -p_0\quad{\rm mod}\ E_4.\]
  This implies that $R$ gives the inverse of the map
  $\partial:{\rm Sol}(E_4(e)\longrightarrow {\rm Sol}(E_4(e'))$,
  and that the corresponding S-value is $p_0$.

\begin{prp}\label{E4red}If one of
  $$e_5,\ e_6,\ e_7,\ e_8$$ is an integer,
  then the equation $E_4$ is reducible.
  \end{prp}

We could not find other shift operator than $\partial$.

\subsection{Reducible cases of $E_4$}\label{E4ap1Red}

\begin{prp}\label{E4redE3}
\[\def\arraystretch{1.1} \begin{array}{cccccccc}
e_5,\dots,e_8=\ \ &\  \cdots   \ &\ -1    \ &\  0     \ &\  1     \ &\  2        \ &\ \cdots\\
  \ &\    [31]    \ &\  [31] \ &\ [31]A0\ &\ [13]A0 \ &\ [13]\ &\   [13]  
\end{array}
\]
  In particular, when $e_7=0,1$, we have
\[\def\arraystretch{1.1} \begin{array}{lcl}E_4(e_7=0)&=&
    E_3(e_1-1,\dots,e_4-1,e_5+1,e_6+1)\circ\partial,\\
    E_4(e_7=1)&=& \partial\circ E_3(e).
\end{array}\]
\end{prp}


\section{$S\!E_4$: a specialization of $E_4$}\label{newSE4} 

$S\!E_4$ is, by definition, a specialization of  $E_4(e)$:
  $$e_1-2e_2+1=e_3-2e_4+1=e_5-2e_6+e_7+e_8,\quad e_1+\cdots+e_7+e_8=4.$$
We parameterize the 5 $(=7-2)$ free parameters by $a,b,c,g, u $ as:
$$\begin{array}{lll}e_1=2 a+2 c+g- u +2,&e_2=a+c- u +1,\\
e_3=2 b+2 c+g- u +2,\quad&e_4=b+c- u +1,\\
e_5= u +1,&e_6= u -a-b-2 c-g-1,\\e_7= u -2 c,&
e_8= u -2-g-2 c-2 b-2 a.\end{array}$$

\subsection{Shift operators of $S\!E_4$}
\begin{thm}\label{shiftopSE4}
   The equation $S\!E_4$ admits shift operators for the shifts $a\to a\pm1$, $b\to b\pm1$, $c\to c\pm1$, $g\to g\pm2$, and    $ u \to  u \pm1$.
\end{thm}
The shift operators for the shifts of $(a,b,c,g)$ can be obtained  exactly in the same way as in \S \ref{shiftop_SE3_to_SE6} from those of $S\!E_3$:
$$S\!E_4=\partial^ u \cdot \partial S\!E_3 \cdot\partial^{- u },\quad R_1=\partial^ u \cdot \partial^2P_{ap}\cdot\partial^{- u },\dots$$
Or by solving directly the shift equation $(EPQE)$:
 let $E_s$ denote a shifted operator and compute
 $$    X:=E_s\circ P-Q\circ S\!E_4 $$
  assuming $P$ and $Q$ are polynomials of $\partial$ up to degree 3
  with coefficients of $x$ with degree at most 5; Let $V$ be the
  set of coefficients.
  Develop $X$ as polynomials of $x$ and $\partial$, and let $S$ be the set   of coefficients. Then solve the set of equations $S$ relative to $V$.

For example, for the shift $a\to a-1$, this set yields
shit operators as:
$$\begin{array}{ll}
P_{a-}&=x(x-1)^2\partial^3+(x-1)((e_6-2e_4+4)x+e_2-2)\partial^2
  + ((1-e_3)(1-e_4)x\\& - c_p(x-1))\partial+e_6e_8(e_1+e_7+ u -1):\\
Q_{a-}&=x(x-1)^2\partial^3+(x-1)((e_6-2e_4+7)x+e_2-3)\partial^2
   + ((2-e_3)(2-e_4)x\\& - c_q(x-1))\partial+(e_6+1)(e_8+2)(e_1+e_7+ u -1):
  \end{array}$$
where
$$\begin{array}{ll}
  c_p&:=-2 u^2+(g+5c+3b+2a) u +4a^2+(6+2b+2c+4g)a-2c^2\\
&+(g+3-2b)c+g^2   +(3+b)g+2+3b:\\
  c_q&:=-2 u^2+(g+5c+3b-4+2a) u +4a^2+(4+2c+4g+2b)a-2c^2\\
&+(g+6-2b)c+g^2   +(2+b)g-2+4b:
  \end{array}$$

For the reverse shift $a\to a+1$, we follow 
the arguments in \S \ref{secshift} by relying on Lemma \ref{syseq}.
Similar arguments apply to the cases
  $b \to b\pm1$, $c\to c\pm1$ and $g\to g\pm2$.

\begin{remark}
The explicit expressions of these shift operators are not given in this
paper, while those are listed in the file named {\it SE4PQ.txt}
in the site given in Introduction.
\end{remark}

\subsection{S-values and reducibility conditions of $S\!E_4$}

Thanks to the explicit expressions of the shift operators, we get
\begin{thm}
$$\begin{array}{ll}
Sv_{u+} :=&( u +1) (1+a+2 c+g+b- u ) (2 c- u ) (2+2 b+2 a+g+2 c- u ):\\

Sv_{a+}:=&(a+1) (2 a+g+2) (2 a+2 c+g+3) (a+b+c+g+2) \\
&(- u +3+g+2 c+2 b+2 a) (a+b+2 c+g- u +2) (4+2 b+2 a+g+2 c- u ):\\
  
Sv_{c+}:=& (2 c+g+2) (2 b+2 c+g+3) (2 a+2 c+g+3) (a+b+c+g+2)\\
&(2 c+2- u ) (2 c- u +1) (4+2 b+2 a+g+2 c- u ) (3+a+2 c+g+b- u )\\
& (a+b+2 c+g- u +2) (- u +3+g+2 c+2 b+2 a):\\

Sv_{g+}:=&-(g+1) (g+2 c+2) (2 b+g+2) (2 b+2 c+g+3) (2 a+g+2)\\
& (2 a+2 c+g+3) (g+3+c+b+a) (a+b+c+g+2)(- u +3+g+2 c+2 b+2 a)\\
& (3+a+2 c+g+b- u ) (a+b+2 c+g- u +2) (4+2 b+2 a+g+2 c- u ):
  \end{array}$$
  \end{thm}

\begin{thm} \label{redcondSE4}
  If one of the values
\begin{gather*}
-a,\ -b,\ (g+1)/2,\ a+g/2+1,\ b+g/2+1,\ c+g/2+1, \\[2mm]
a+b+c+g+2,\quad  a+c+g/2+3/2,\quad b+c+g/2+3/2,\\[2mm] 
e_5,\ e_6,\ e_7,\ e_8
\end{gather*}
is an integer, then the equation $S\!E_4$ is reducible.
\end{thm}

\subsection{Reducible cases of $S\!E_4$} 

Added to Proposition \ref{E4redE3}, we have
\begin{prp}\label{factorSE5cont}(1) Let $\Delta$ be one of $\{-a,-b,a+g/2+1,b+g/2+1,a+c+g/2+3/2,b+c+g/2+3/2.$ Then 
\[ \def\arraystretch{1.1} \begin{array}{rcccccc} 
 \Delta=\ &\ \cdots\ &\   -1,\ &\ 0,   \ &\ 1,    \ &\ 2,    \ &\ \cdots\\
    \ &\ \cdots\ &\ [3,1]\ &\ [3,1]A0\ &\ [1,3]A0\ &\ [1,3]\ &\ \cdots\\
\end{array}\]
When the factors are free of apparent singularities,
the factor $[3]$ is essentially equal to a specialization of ${}_3E_2$.

\par\noindent
(2) Let $\Delta$ be one of $\{(g+1)/2,c+g/2+1,a+b+c+g+2\}$. Then
\[\begin{array}{rcccccc} 
  \Delta=\ &\ \cdots\ &\   -1,\ &\ 0,   \ &\ 1,    \ &\ 2,    \ &\ \cdots\\
    \ &\ \cdots\ &\ [2,2]\ &\ [2,2]A0\ &\ [2,2]A0\ &\ [2,2]\ &\ \cdots\\
\end{array}\]
\end{prp}

\subsection{Summary: reducibility conditions for $S\!E_6$,
  $S\!E_5$, $S\!E_4$, and $S\!E_3$}

The reducibility conditions for $S\!E_6,\dots, S\!E_3$ in terms of $\{a,b,c,g\}$, stated in this paper, can be tabulated as follows:

$$\begin{array}{ccccc}
\quad &\quad S\!E_6\quad &\quad S\!E_5\quad &\quad S\!E_4
\quad &\quad S\!E_3\\[2mm]
a          &\{2,4\}&\{1,4\}&\{1,3\}&\{1,2\}\\[2mm]
b          &\{2,4\}&\{1,4\}&\{1,3\}&\{1,2\}\\[2mm]
c          &\{2,4\}&\{2,3\}&-  &\{1,2\}\\[2mm]
1/2+g/2    &\{2,4\}&\{2,3\}&\{2,2\}&\{1,2\}\\[2mm]
a+g/2      &\{2,4\}&\{2,3\}&\{1,3\}&\{1,2\}\\[2mm]
b+g/2      &\{2,4\}&\{2,3\}&\{1,3\}&\{1,2\}\\[2mm]
c+g/2      &\{2,4\}&\{1,4\}&\{2,2\}&\{1,2\}\\[2mm]
a+b+c+g    &\{2,4\}&\{2,3\}&\{2,2\}&\{1,2\}\\[2mm]
a+b+c+g/2  &\{2,4\}&\{1,4\}&-  &\{1,2\}\\[2mm]
a+b+g/2+1/2&- &\{1,4\}&-  &-\\[2mm]
a+c+g/2+1/2&-  &-  &\{1,3\}&-\\[2mm]
b+c+g/2+1/2&-  &-  &\{1,3\}&-
\end{array}$$
\par\medskip\noindent
For example the third line reads: when $c\in\mathbb{Z}$, the equations $S\!E_6$, $S\!E_5$ and $S\!E_3$ are reducible of type $\{2,4\}$,  $\{2,3\}$ and $\{1,2\}$, respectively, but no factorization is found for $S\!E_4$.    

\newpage

\bibliographystyle{amsalpha}

\begin{thebibliography}{AA}
\bibitem[1]{AK}Aomoto,~K. and Kita,~M., Theory of Hypergeometric Functions,
  1994, pp.x+355, Springer;
  (English edition translated by K. Iohara and with an appendix by
  T. Kohno, in
  Springer Monographs in Mathematics, 2011, pp. xvi+317).

\bibitem[2]{DR} Dettweiler,~M. and Reiter,~S., 
 An algorithm of Katz and its application to the inverse Galois problem, 
 J. Symbolic Comput., {\bf 30} (2000), 761--798.
  
\bibitem[3]{DF}Dotsenko,~V.~S. and Fateev,~V.~A., Conformal algebra and  multipoint correlation functions in 2D statistical models,  Nucl. Phys. B240[FS12], 312--348 (1984)

\bibitem[4]{EOY} Ebisu,~A., Ochiai,~H. and Yoshida,~M., Study of reducible types of the generalized hypergeometric equation ${}_3E_2$ and the Dotsenko-Fateev equation, preprint (2025)

\bibitem[5]{Hara} Haraoka,~Y., On Oshima's middle convolution,   Josai Math. Monographs 12 (2020), 19-51.

\bibitem[6]{HOSY1} Haraoka,~Y., Ochiai,~H., Sasaki,~T. and Yoshida,~M., Fuchsian equations of order 3,...,6 with three singular points and an accessory parameter, to appear in Funkcial. Ekvac.
  
\bibitem[7]{Mim} Mimachi.~K., Monodromy representation associated with the Fuchsian differential equation of order 3 derived by Dotsenko-Fateev and its irreducibility, Commun. Math. Phys. Digital Object Identifier (DOI) https://doi.org/10.1007/s00220-024-04942-7 (2024).

\bibitem[8]{Osh} Oshima,~T., {\sl Fractional Calculus of Weyl Algebra
  and Fuchsian Differential Equations}, Memoirs of Math. Soc. Japan,
  vol. 28(2012), pp. xix+203.

\bibitem[9]{ST} Sasai,~T. and Tsuchiya,~S., On a fourth order Fuchsian
  differential equation of Okubo type, Funkcial. Ekvac. 34(1991), 211-221.

\end{thebibliography}

\noindent 
Yoshishige Haraoka

Josai University, Sakado 350-0295, Japan

haraoka@kumamoto-u.ac.jp

\medskip \noindent
Hiroyuki Ochiai

Department of Mathematics, Kyushu University, Fukuoka 819-0395, Japan 

ochiai@imi.kyushu-u.ac.jp

\medskip \noindent 
Takeshi Sasaki

Kobe University, Kobe 657-8501, Japan 

yfd72128@nifty.com

\medskip \noindent
Masaaki Yoshida

Kyushu University, Fukuoka 819-0395, Japan 

myoshida1948@jcom.home.ne.jp

\end{document}